\documentclass[12pt,a4paper]{article}
\usepackage{latexsym,amssymb,amsmath,mathrsfs,amsthm}

\usepackage{sectsty}
\usepackage{color}
\usepackage{bm}
\usepackage[T1]{fontenc}
\usepackage[anchorcolor=blue,%
 bookmarks=true,%
 bookmarksnumbered=true,%
]{hyperref}
\usepackage{pst-all}
\usepackage{graphicx}

\allsectionsfont{\normalsize\sc\center}

\newtheorem{thm}{Theorem}[section]

\newtheorem{cor}[thm]{Corollary}
\newtheorem{lem}[thm]{Lemma}

\newtheorem{defn}[thm]{Definition}
\newtheorem{remark}[thm]{Remark}
\newtheorem{example}[thm]{Example}

\makeatletter \@addtoreset{equation}{section} \makeatother

\renewcommand{\P}{\mathbb{P}}
\newcommand{\E}{\mathbb{E}}

\newcommand{\R}{\mathbb{R}}

\newcommand{\N}{\mathbb{N}}

\DeclareMathOperator{\Ric}{Ric}

\renewcommand{\d}{\mathrm{d}}

\newcommand{\1}{{\bf 1}}
\newcommand{\eps}{{\varepsilon}}

\title{\large\bf Liouville theorem for \boldmath$V$-harmonic maps under
non-negative \boldmath$(m, V)$-Ricci curvature for non-positive $m$}
\author{Kazuhiro Kuwae\thanks{Department of Applied Mathematics, Fukuoka University,
Fukuoka 814-0180, Japan ({\sf kuwae@}  {\sf fukuoka-u.ac.jp}). Supported in part by JSPS Grant-in-Aid for Scientific Research (KAKENHI) 22H04942 and by funds (No.: 197004 and 215001) from the Central Research Institute of Fukuoka University.},
\ \ \ \
Songzi Li\thanks{School of Mathematics, Renmin University of China, 59, Zhongguancun Da Jie, Beijing, 100872, China ({\sf sli@ruc.edu.cn}). Supported by NSFC No.~11901569 and by fund (No.:2018030249) from Renmin University of China.},
\ \ \ \
Xiang-Dong Li\thanks{Academy of Mathematics and Systems Science, Chinese Academy of Sciences, 55, Zhongguancun East Road, Beijing, 100190, China ({\sf xdli@amt.ac.cn}), and School of Mathematical Sciences, University of Chinese Academy of Sciences, Beijing, 100049, China. Supported by National Key R$\&$D Program of China (No.~2020YF0712700), NSFC No.~11771430, No.~11688101,
and Key Laboratory RCSDS, CAS, No.~2008DP173182.}
\ \ and\ \
Yohei Sakurai\thanks{Department of Mathematics, Saitama University, 255 Shimo-Okubo, Sakura-ku, Saitama-City, Saitama, 338-8570, Japan ({\sf ysakurai@rimath.saitama-u.ac.jp}). Supported in part by JSPS Grant-in-Aid for Scientific Research on Innovative Areas \lq\lq Discrete Geometric Analysis for Materials Design\rq\rq\; 17H06460).
}
}
\date{}
\begin{document}
\maketitle
\begin{abstract}
Let $V$ be a $C^1$-vector field on an $n$-dimensional complete Riemannian manifold $(M, g)$. 
We prove a Liouville theorem for 
$V$-harmonic maps satisfying various growth conditions from complete Riemannian manifolds with  non-negative $(m, V)$-Ricci curvature for $m\in\,[\,-\infty,\,0\,]\,\cup\,[\,n,\,+\infty\,]$ into Cartan-Hadam\-ard manifolds, which extends Cheng's Liouville theorem proved in \cite{SYCheng} for sublinear growth
harmonic maps  from complete Riemannian manifolds with  non-negative Ricci curvature into Cartan-Hadamard manifolds.
We also prove a Liouville theorem for
$V$-harmonic maps  from complete Riemannian manifolds with  non-negative $(m, V)$-Ricci curvature for $m\in\,[\,-\infty,\,0\,]\,\cup\,[\,n,\,+\infty\,]$ into regular geodesic balls of Riemannian manifolds 
with positive upper sectional curvature bound, 
which extends the  
results of 
Hildebrandt-Jost-Wideman~\cite{HildtJostWidemann}
and Choi~\cite{Choi}. 
Our probabilistic proof of Liouville theorem for several growth $V$-harmonic maps into Hadamard manifolds enhances an incomplete argument in \cite{Staff:Liouville}.
Our results 
extend  the results due to Chen-Jost-Qiu~\cite{ChenJostQiu} and Qiu~\cite{Qiu} in the case of $m=+\infty$
on the  Liouville theorem  for  bounded $V$-harmonic maps from complete Riemannian manifolds with non-negative $(\infty, V)$-Ricci curvature into regular geodesic balls of
Riemannian manifolds with positive sectional curvature upper bound. Finally, we establish a connection between the Liouville property of $V$-harmonic maps and the recurrence property of $\Delta_V$-diffusion processes on manifolds.  
Our results are new even in the case $V=\nabla f$ for $f\in C^2(M)$. 
\end{abstract}

 

\section{Statement of Main Theorem}
Let $(M,g,f)$ denote an $n$-dimensional weighted Riemannian manifold, namely, $(M,g)$ is an $n$-dimensional complete Riemannian
manifold, and $f\in C^{\infty}(M)$. {\color{black}{Let ${\sf d}:={\sf d}_g$ denote the Riemannian distance and 
let $B_r(p):=\{x\in M\mid {\sf d}(p,x)<r\}$ denote the open ball with center $p\in M$ and radius $r>0$.}} 
Throughout this paper,
we always assume $n\geq2$. Let
$\Delta_f:=\Delta-\langle\nabla f,\nabla\cdot\rangle$ be a diffusion operator called \emph{$f$-Laplacian} or \emph{Witten Laplacian}, which is
symmetric with respect to the measure $\mu(\d x):=e^{-f(x)}\nu_g(\d x)$, where
$\nu_g$ is the volume measure of $(M,g)$.
For $m\in\,[\,-\infty,\,+\infty\,]$, the {\it $m$-weighted Ricci curvature} for $(M,g,f)$ or $\Delta_f$ is defined as following (\cite{BE1}, \cite{Lich},  \cite{Xdli:Liouville}):
\begin{align*}
\Ric_f^m:=\Ric_g+{\rm Hess}\,f-\frac{\d f\otimes \d f}{m-n},
\end{align*}
where when $m=\pm\infty$,
the last term is interpreted as the limit $0$,
and when $m=n$,
we only consider a constant function $f$, and
set $\Ric_f^n:=\Ric_g$.

Under the classical curvature condition
\begin{align*}
\Ric_f^m\geq Kg
\end{align*}
for $K\in \mathbb{R}$ and $m\in [\,n,\,+\infty\,]$,
comparison geometry has been developed by many authors (see e.g. \cite{BE1},
\cite{BQ}, \cite{FLZ}, \cite{FLL},  \cite{Xdli:Liouville}, \cite{Lo},
\cite{Qi}, \cite{Qian}, \cite{WW},
and references therein). In this case, the parameter $m$ is 
interpreted as an upper bound for dimension $n$.
In recent years,
comparison geometry for the complementary case of $m\in ]-\infty,n\,[$ has begun to be investigated (see e.g., \cite{KL}, \cite{KSa}, \cite{KS}, \cite{Lim}, \cite{LMO:CompaFinsler}, \cite{MRS:2021}, \cite{Mai1}, \cite{Mai2}, \cite{Mineg}, \cite{Oh<0}, \cite{Sak}, \cite{Wy}, \cite{WyYero}).
In these papers, $m$ is considered as a parameter less than $1$ or non-positive, and it is called 
{\it effective dimension}. In this case, $m$ might not seem an upper bound for dimension $n$ if we keep the usual order for real line $\R$. In view of the coincidence ${\Ric}_V^{\infty}={\Ric}_V^{-\infty}$, if we change the order like $n\leq m_1\leq +\infty\text{\;\lq\lq$=$\rq\rq}\!-\infty\leq m_2\leq 1$, then the parameter $m_2$ can be understood as an upper bound for the dimension $n$. 
Wylie-Yeroshkin \cite{WyYero}  introduced a curvature condition.
\begin{align*}
\Ric_f^1\geq (n-1)\kappa e^{-\frac{4f}{n-1}}g
\end{align*}
for $\kappa \in \mathbb{R}$ in view of the study of weighted affine connection,
and obtained a Laplacian comparison theorem, Bonnet-Myers type diameter comparison theorem, Bishop-Gromov type volume comparison theorem, and rigidity results for the equality cases.
The first and the third named authors \cite{KL} have provided a generalized condition
\begin{align*}
\Ric_f^m\geq (n-m)\kappa e^{-\frac{4f}{n-m}}g
\end{align*}
with $m\in\, ]-\infty,\,1\,]$,
and extended the comparison geometric results in \cite{WyYero}.

In~\cite{Yau:Harmonic}, Yau developed the method of gradient estimate and proved the strong or bounded Liouville theorem on complete Riemannian manifolds. More precisely, let $(M,g)$ be  a complete Riemannian manifold with Ricci curvature bounded from below by $-K$, i.e., ${\Ric}_g\geq -Kg$, where $K\geq0$ is a constant, and $u$ be a harmonic function (i.e. a solution to $\Delta u=0$) which is  bounded from below, then the following gradient estimate holds
\begin{align}
|\nabla u|\leq\sqrt{(n-1)K}\left(u-\inf_Mu \right).\label{eq:YauGrad}
\end{align}
Taking $K=0$, the gradient estimate \eqref{eq:YauGrad}
implies the strong or bounded Liouville theorem, namely, every non-negative or bounded harmonic function on a complete Riemannian manifold with non-negative Ricci curvature must be constant.
In \cite{Xdli:Liouville}, the third named author 
extended Yau's gradient estimate to $f$-harmonic functions, namely, let $(M, g, f)$ be  a complete Riemannian manifold with ${\Ric}_f^m\geq -Kg$, where $K\geq 0$ and $m\in[\,n\,,+\infty\,[$ are two constants,   and
 $u$ be a solution to $\Delta_f u=0$ which is bounded from below, then
\begin{align}
|\nabla u|\leq\sqrt{(m-1)K}\left(u-\inf_Mu \right).\label{eq:YauGradWittenLaplacian}
\end{align}
Taking $K=0$, the gradient estimate \eqref{eq:YauGradWittenLaplacian}
implies an extension of Yau's strong or bounded Liouville theorem, namely, every non-negative or bounded $\Delta_f$-harmonic function on a complete Riemannian manifold with non-negative $m$-dimensional Bakry-Emery Ricci curvature (i.e.,
$\Ric_f^m\geq 0$) must be constant.

In the case $m=+\infty$, Brighton \cite{Brighton} proved the following
Liouville Theorem for bounded $\Delta_f$-harmonic functions
as a corollary of gradient estimate for $\Delta_f$-harmonic functions under non-negative lower bounds of ${\Ric}_f^{\infty}$.

\begin{thm}[{\cite[Corollary~2]{Brighton}}]\label{thm:Brighton}
Suppose ${\Ric}_f^{\infty}\geq 0$. If $u$ is a bounded $\Delta_f$-harmonic function, then $u$ is constant.
\end{thm}

Theorem \ref{thm:Brighton} can be derived from the Hamilton type Harnack inequality. In \cite{Li-SPA16}, the third named author proved  the Hamilton type dimension free Harnack inequality for  bounded positive solutions to the heat equation of $\Delta_f$. More precisely, let $u$ be a positive solution to the
heat equation $\partial_t u=\Delta_f u$ with $u\leq A$ for some constant $A>0$. Suppose that $\Ric_{f}^{\infty}\geq -Kg$ for some constant $K\geq 0$. Then the following Harnack inequality holds
\begin{eqnarray}
|\nabla \log u|^2\leq {2K\over 1-e^{-2Kt}}\log(A/u),  \label{Li-Harnack}
\end{eqnarray}
which improves Hamilton type dimension free Harnack inequality for  bounded positive solutions to the heat equation of $\Delta_f$, i.e.,
\begin{eqnarray}
|\nabla \log u|^2\leq \left({1\over t}+K\right)\log(A/u). \label{Hamilton-Harnack}
\end{eqnarray}
In particular, when $K=0$, and letting $t\rightarrow \infty$,  we can derived that all  bounded $\Delta_f$-harmonic functions must be constant. That is to say, we can recapture Brighton's result from Harnack inequalities $(\ref{Li-Harnack})$ and
$(\ref{Hamilton-Harnack})$  for the heat equation of
$\Delta_f$. In  \cite{LL-JFA18}, 
the second and third named authors 
proved the Li-Yau-Hamilton
type Harnack inequality for positive solutions to the heat equation $\partial_t u=\Delta_f u$  on complete Riemannian manifolds with  ${\Ric}_f^m\geq -(m-1)Kg$ with 
$m\in[\,n,\,+\infty\,[$.
 In \cite{LL-AJM18}, the second and third named authors proved $(\ref{Li-Harnack})$ for bounded positive solutions to the heat equation $\partial_t u=\Delta_f u$  on complete $(K, \infty)$-super Ricci flow. See also \cite{LL-AJM18, LL-SCM19} for the Li-Yau type and the Li-Yau-Hamilton
type Harnack inequalities for positive solution to the heat equation $\partial_t u=\Delta_f u$ on Ricci flow and variant of super Ricci flows. Note that
Hua \cite{Hua2019} also provided another proof of Theorem \ref{thm:Brighton}.

On the other hand, S.~Y.~Cheng~\cite{SYCheng} generalized Yau's gradient
estimate to harmonic maps. Recall that
a  smooth map $u: M\rightarrow
 N$ between two Riemannian manifolds $(M, g)$ and $(N, g)$ is said to be a \emph{harmonic map} if
$\tau(u)=0$, where $\tau(u)$ is the tension field of $u$ and
$\d u:TM\to TN$ defined by $(\d u)_x:T_xM\to T_{u(x)}N$ is the differential of $u$,
see Section $3$ below or \cite{S-Y} and related references.

\begin{thm}[{\cite[Theorem]{SYCheng}}]\label{thm:Cheng}
Let $(M,g)$ be a complete Riemannian manifold with $\Ric_g\geq -Kg$, $K\geq0$, and $N$ a Cartan-Hadamard manifold, i.e., a complete connected and  simply connected Riemannian manifold having non-positive sectional curvature. Let $u:M\to N$ be a harmonic map. Assume $u(B_a(p))\subset B_b(o)$ for some $p\in M$ and $o\in N$ and some $a,b>0$. Then we have
\begin{align}
\sup_{B_{a/2}(p)}|\d u|^2\leq C_n\cdot\frac{b^4}{a^4}\max\left\{\frac{Ka^4}{\beta},\frac{a^2(1+Ka)}{\beta},\frac{a^2b^2}{\beta^2} \right\},\label{eq:Cheng}
\end{align}
where $C_n$ is a constant depending only on $n$ and $\beta=b^2-\sup_{B_a(p)}{\color{black}{{\sf d}}}_N^2(u,o)$. {\color{black}{Here ${\sf d}_N$ denotes the Riemannian distance of $(N,h)$. }} 
\end{thm}

As a corollary, S.~Y.~Cheng~\cite{SYCheng} established a Liouville theorem for sublinear growth harmonic maps:

\begin{cor}[{\cite[Liouville theorem for harmonic maps]{SYCheng}}]\label{cor:Cheng}
Let $(M,g)$ be a complete Riemannian manifold with $\Ric_g\geq 0$, and $N$ a Cartan-Hadamard manifold, i.e., a complete connected and simply connected Riemannian manifold having non-positive sectional curvature. Let $u:M\to N$ be a  harmonic map.
If $u$ satisfies the sublinear growth
$$
\varlimsup_{a\to\infty}\frac{\sup_{x\in B_a(p)}{\color{black}{{\sf d}}}_N(u(x),o)}{a}=0
$$
for some $p\in M$ and $o\in N$.
Then $u$ is a constant map.
\end{cor}

Hildbrandt-Jost-Widemann~\cite{HildtJostWidemann} and Choi~\cite{Choi} further
extended Cheng's work \cite{SYCheng} as the following theorem:

\begin{thm}[{cf.~\cite[(11)]{Choi}}]\label{thm:Choi}
Let $(M,g)$ be a complete smooth Riemannian manifold with $\Ric_g\geq -Kg$, $K\geq0$, 
and $N$ a complete smooth Riemannian manifold having sectional curvature ${\rm Sect}_N\leq\kappa$, $\kappa>0$. 
Let $u:M\to N$ be a harmonic map. Assume $u(M)\subset B_b(o)$ lies inside the cut locus of $o\in N$ and some $b<\pi/(2\sqrt{\kappa})$. Then
\begin{align}
\sup_{B_{a/2}(p)}|\d u|^2&\leq 4\max\left\{\frac{2^{10}}{C_2^2a^2\beta^2},
\frac{2K}{C_2}+\frac{8C_1(1+a)}{C_2a^2}+\frac{64}{C_2a^2} \right\}.
\label{eq:Choi}
\end{align}
for some constants  $C_1,C_2>0$ independent of  $n,\kappa$ and $b$.  
Here $\beta:=b-\sup_{B_a(p)}\varphi\circ u$ and $\varphi(y)=1-\cos \sqrt{\kappa}{\color{black}{{\sf d}}}_N(y,o)$. 

\end{thm}

A geodesic ball $B_R(o)\subset N$ with ${\rm Sect}_N\leq\kappa$ is said to be a \emph{regular geodesic ball} if $B_R(o)\cap {\rm Cut}(o)=\emptyset$
and $R<\pi/2\sqrt{\kappa^+}$ with  $\kappa^+:=\max\{\kappa,0\}$.

\begin{cor}[{\cite[Theorem]{Choi}}]\label{cor:Choi}
Let $(M,g)$ be a complete Riemannian manifold with $\Ric_g\geq 0$,
and $N$ a complete Riemannian manifold with sectional curvature ${\rm Sect}_N\leq\kappa$, $\kappa>0$.
Let $u:M\to N$ be a harmonic map whose range is contained in a regular geodesic ball $B_b(o)$. 
Then $u$ is a constant map.
\end{cor}

The purpose of this paper is to extend Cheng's Liouville theorem for sublinear growth $V$-harmonic maps $u:M\to N$
between Riemannian manifolds $(M,g)$ and $(N,h)$ under the relaxed condition ${\Ric}_V^m\geq0$ for $m\in[-\infty,\,0\,]\,\cup\,[\,n,\,+\infty\,]$ and 
$(N,h)$ is a Cartan-Hadamard manifold, and to extend Hildbrandt-Jost-Widemann and Choi's Liouville theorem for  $V$-harmonic maps $u:M\to N$
between Riemannian manifolds $(M,g)$ and $(N,h)$ under reasonable condition.   Here, not only for the symmetric diffusion operator $\Delta_f$, we work on the more general setting of a Riemannian manifold $(M,g,V)$ with
 a $C^1$-vector field  $V$ which is not necessarily to be the gradient $V=\nabla f$ for some $f\in C^2(M)$.
Such a Riemannian manifold is equipped with canonical $V$-Laplacian and $V$-Ricci curvature. More precisely, the \emph{non-symmetric
 $V$-Laplacian} is defined by
$$
\Delta_V:=\Delta-\langle V,\nabla\cdot\rangle 
$$
and the $m$-dimensional \emph{$V$-Ricci curvature}, or the \emph{$(m, V)$-Ricci curvature},  is defined as follows:
$$
{\Ric}_V^m:={\Ric}_g+\frac12\mathcal{L}_Vg-\frac{V^*\otimes V^*}{m-n},
$$
where $m\in \,[\,-\infty, +\infty]$ and set $V=0$ under $m=n$,  $\mathcal{L}_Vg(X,Y):=\langle \nabla_XV,Y\rangle+\langle\nabla_YV,X\rangle$ is the Lie derivative of $g$ with respect to $V$, and $V^*$ denotes its dual $1$-form.  
We would like to point out that, in the case where $m\in[\,n,\,+\infty\,[$, $\Ric_{V}^m$ has been introduced in previous works
of Qian \cite{Qian} and Bakry-Qian \cite{BQ}.  In \cite{Xdli:Liouville} (Section 8.7, p.~1347),  for $V=-B$,
the third named author 
used the notation $\Ric_{m, n}^S(L)=\Ric_g-\nabla^S B-{B\otimes B\over m-n}$ (instead of $\Ric_V^m$) to denote the $m$-dimensional Bakry-Emery Ricci curvature for non-symmetric elliptic diffusion
operator $L=\Delta+B$, where $\nabla^S B(\xi, \eta)={1\over 2}\{\langle  \nabla_\xi B, \eta\rangle+\langle \nabla_\eta B, \xi\rangle\}$.
Moreover, it has been pointed out in  \cite{Xdli:Liouville} that  it would be very possible to extend the
various Liouville theorems and Harnack inequalities to non-symmetric diffusion operators. Indeed, as the Laplacian comparison theorem holds for $L=\Delta+B$ under the condition $\Ric_{m, n}^S(L)\geq Kg$ for $K\in \mathbb{R}$
and $m\in[\,n,\,+\infty\,[$, see Bakry-Qian \cite{BQ}, the proofs of
the strong Liouville theorem and the Li-Yau Harnack inequality for symmetric elliptic diffusion operators in  \cite{Xdli:Liouville}  can be extended to non-symmetric diffusion operators. For the detail, see \cite{YiLi:2015}.

In the case of $V=\nabla f$, it coincides with ${\Ric}_f^m$. For $m_1\in\,]-\infty,1\,]$ and $m_2\in\,[\,n,+\infty\,[$ we have the following order
\begin{align}
{\Ric}_V^1\geq{\Ric}_V^{m_1}\geq{\Ric}_V^{-\infty}={\Ric}_V^{\infty}\geq{\Ric}_V^{m_2}.\label{eq:order}
\end{align}
So the condition ${\Ric}_V^m\geq Kg$ is weaker than ${\Ric}_V^{\infty}\geq Kg$ for $m\in\,]\,-\infty,\,1\,]$ and $K\in\R$. 
\newline

We now describe more detail of our setting. Let $(M, g)$ and $(N,h)$ be two complete smooth Riemaniann manifolds. A smooth map $u: M\rightarrow
 N$ is said to be \emph{$V$-harmonic} if
$\tau_V(u):=\tau(u)-(\d u)(V)=0$. 
When $V=\nabla f$ for $f\in C^2(M)$, any $V$-harmonic map is called the \emph{$f$-harmonic map}. The notion of $V$-harmonic map covers the various notions of harmonicity for maps, e.g., Hermitian harmonic maps, Weyl harmonic maps and affine harmonic maps (see \cite{ChenJostWang} for these notions on harmonicity). In \cite[Theorem~2]{ChenJostQiu},  
Chen-Jost-Qiu proved the bounded Liouville property for $V$-harmonic maps from complete Riemannian manifolds with ${\Ric}_V^{\infty}\geq0$  into regular geodesic ball with 
${\rm Sect}_N\leq\kappa$ for $\kappa\geq0$ under the sublinear growth condition for $V$. 
Later, Qiu~\cite[Theorem~2]{Qiu}  
also proved the bounded Liouville property as in \cite[Theorem~2]{ChenJostQiu} 
without assuming the sublinear growth condition for $V$. 
We will prove not only the bounded Liouville property for $V$-harmonic maps from complete Riemannian manifolds 
into regular geodesic ball but also Liouville property for sublinear growth $V$-harmonic map into Cartan-Hadamard manifold $(N,h)$ under ${\Ric}_V^m\geq0$ with $m\in[\,-\infty,\,0\,]\,\cup\,[\,n,+\infty\,]$.
\\\quad\\
 Hereafter, we always fix points $p\in M$ and $o\in N$, and also $m\in\,[\,-\infty,\,0\,]\,\cup\,[\,n,\,+\infty\,[$.

\begin{defn}
{\rm
A map $u:M\to N$ is said to be of \emph{sublinear growth} if
\begin{align*}
\varlimsup_{a\to\infty}m_u(a)/a=0,
\end{align*}
for some/any $o\in N$, 
where $m_u(a):=\sup_{x\in B_a(p)}{\color{black}{{\sf d}}}_N(u(x),o)$, i.e., $m_u(a)=o(a)\;\; (a\to\infty)$. 
}
\end{defn}

We prepare several growth conditions for a smooth map $u:M\to N$:  
\begin{enumerate}
\item[{\bf {$\;$}(G1)}] $u$ is of sublinear growth, i.e., $m_u(a)=o(a),\quad (a\to\infty)$.  
\item[{\bf {$\;$}(G2)}] $m_u(a)=o(\sqrt{a}),\quad (a\to\infty)$.
\item[{\bf {$\;$}(G3)}] $m_u(a)=o(\sqrt{\log a}),\quad (a\to\infty)$.
\end{enumerate} 
It is easy to  see that $\text{\bf (G1)}$ (resp.~$\text{\bf (G2)}$) is weaker than $\text{\bf (G2)}$ 
(resp.~$\text{\bf (G3)}$). Note here that any bounded map $u:M\to N$ satisfies $\text{\bf (G3)}$, i.e., the boundedness of $u$ is stronger than any {\bf(G{\boldmath$i$})}.

Fix a point $p\in M$. 
We set $v(r):=\sup_{x\in B_r(p)}|V|_x$ for $r>0$ and $v(0):=|V|_p$. Then $v$ is a non-negative continuous function on $[0,+\infty[$ and 
$|V|_x\leq v(r_p(x))$ holds for $x\in M$. {\color{black}{Here $r_p(x):={\sf d}(x,p)$ is the radial function from a referene point $p\in M$.}}
Now we consider the following conditions: 
\begin{enumerate}
\item[{\bf (A1)}] $v\in L^1([0,+\infty[)$, or $\langle V,\dot{\gamma}_t\rangle_{\gamma_t}\geq0$ a.e.~$t$ holds for any unit speed geodesic $\gamma$ starting from $p$;
\item[{\bf (A2)}] There exists $D\geq1$ such that $\int_0^rv(s)\d s\leq \frac{n-m}{4}\log D(1+r)$;
\item[{\bf (A3)}] There exists $D\geq1$ such that $\int_0^rv(s)\d s\leq \frac{n-m}{4}\log D(1+r^2)$.
\end{enumerate}
The condition $\text{\bf (A2)}$ (resp.~$\text{\bf (A3)}$) with $D=1$ is satisfied if $v(r)\leq\frac{n-m}{4(1+r)}$ 
(resp.~$v(r)\leq \frac{(n-m)r}{2(1+r^2)}$) for $r\geq0$. 
\medskip

Our first main result extends Cheng's Liouville theorem to sublinear growth $V$-harmonic maps on manifolds with ${\Ric}_V^m\geq 0$ and ${\rm Sect}_N\leq 0$.

\begin{thm}\label{thm:Liouville1} Let $(M, g, V)$ be a complete Riemaniann manifold with 
 ${\Ric}_V^m\geq0$ for some $C^1$-vector field $V$ and some constant $m\in\,[-\infty,\,0\,]
 \,\cup\,[\,n,\,+\infty\,]$,  and
$(N,h)$ is a Cartan-Hadamard manifold, i.e., $N$ is a complete connected and simply connected Riemannian manifold with ${\rm Sect}_N\leq0$. Let $u:M\to N$ be a $V$-harmonic map.
For each $i=1,2,3$, we suppose {\bf (A{\boldmath$i$})} for $p\in M$ and {\bf (G{\boldmath$i$})} 
for $u$ provided $m\in[-\infty,\,0\,]\,\cup\,\{+\infty\}$, otherwise, we only suppose 
{\bf (G1)}. 
Then $u$ is a constant map. 
\end{thm}

For each $i=1,2,3$, 
we consider the following condition ${\bf (B\text{\boldmath$i$})}$ on the shape of Laplacian comparison. 
\begin{enumerate}
\item[{\bf (B\text{\boldmath$i$})}] There exists $D\geq1$ such that $r_p(x)\Delta_Vr_p(x)\leq D(1+r_p(x)^{i-1})$ for $x\notin {\rm Cut}(p)$.
\end{enumerate}
{\color{black}{Note here that the radial function $x\mapsto r_p(x)$ is secondly differentiable at $x\notin {\rm Cut}(p)$.}} 
The condition {\bf (A\text{\boldmath$i$})} implies 
{\bf (B\text{\boldmath$i$})} for each $i=1,2,3$ (see Theorem~\ref{thm:(A)0} below). 
\begin{remark}
{\rm 
The statement of Theorem~\ref{thm:Liouville1} remains valid by replacing 
{\bf (A{\boldmath$i$})} with {\bf (B{\boldmath$i$})} for each $i=1,2,3$ (see 
the proof of Theorem~\ref{thm:(A)0}). 
}
\end{remark}

Our second main result extends Hildbrandt-Jost-Widemann and Choi's Liouville theorem as follows: 

\begin{thm}\label{thm:Liouville2} Let $(M, g, V)$ be a complete Riemaniann manifold with ${\Ric}_V^m\geq0$ for some $C^1$-vector field $V$ and some constant $m\in\,[-\infty,\,0\,]
\,\cup\,[\,n,+\infty\,]
$, and $N$ be a complete Riemannian manifold with ${\rm Sect}_N\leq\kappa$ with $\kappa>0$. 
Suppose {\bf (A3)} for  $p\in M$ when $m\in[-\infty,\,0\,]\,\cup\,\{+\infty\}$.   
Let $u:M\to N$ be a $V$-harmonic map whose range is contained in a regular geodesic ball $B_R(o)$.
Then $u$ is a constant map.
\end{thm} 

\begin{cor}[{Chen-Jost-Qiu~\cite[Theorem~2]{ChenJostQiu} and 
Qiu~\cite[Theorem~2]{Qiu}}]\label{cor:ChenJostQiu}\quad\\
Let $(M, g, V)$ be a complete Riemaniann manifold with ${\Ric}_V^{\infty}\geq0$ for some $C^1$-vector field $V$, and $N$ be a complete Riemannian manifold with ${\rm Sect}_N\leq\kappa$ with $\kappa>0$. 
Let $u:M\to N$ be a $V$-harmonic map whose range is contained in a regular geodesic ball $B_R(o)$.
Then $u$ is a constant map.
\end{cor}
\begin{cor}\label{cor:New}
Let $(M, g, V)$ be a complete Riemaniann manifold with ${\Ric}_V^m\geq0$ for some $C^1$-vector field $V$ and some constant $m\in\,]-\infty,\,0\,]\,\cup\,[\,n,+\infty\,[
$, and $N$ be a complete Riemannian manifold with ${\rm Sect}_N\leq\kappa$ with $\kappa>0$. Suppose that $V$ is bounded. 
Let $u:M\to N$ be a $V$-harmonic map whose range is contained in a regular geodesic ball $B_R(o)$.
Then $u$ is a constant map.
\end{cor}

We now discuss the connection between the Liouville property of harmonic maps and the recurrence property of diffusion processes on manifolds. In the case of harmonic functions, it is well-known that all bounded subharmonic functions must be constant if and only if  the Brownian motion on manifold is recurrent. This  remains true for general $V$-subharmonic functions and the $\Delta_V$-diffusion 
process on manifolds.  To help the reader who is not familiar in this result, let us recall some notions and facts in
potential theory.  Let ${\bf X}^V=(\Omega, X_t,{\P}_x)$ be the $\Delta_V$-diffusion process associated to
the $V$-Laplacian $\Delta_V$ and $\zeta:=\inf\{t>0\mid X_t\notin M\}$ its life time (see Section~\ref{sec:diffusion} below). A Borel function $g$ on $M$ is said to be \emph{excessive} if
$g$ is non-negative on $M$ and
${\E}_x[g(X_t):t<\zeta]\leq g(x)$ for any $t>0$ and $\lim_{t\to0}
{\E}_x[g(X_t):t<\zeta]=g(x)$  and $x\in M$.
${\bf X}^V$ is said to be \emph{recurrent} if
any excessive function is constant (see \cite[(2.4) Proposition]{Get:TranRec}). Our third main result is the following:

\begin{thm}\label{thm:Liouville3}
Let $(M,g)$, $(N,h)$ be two complete  Riemannian manifolds.
Suppose that the $\Delta_V$-diffusion process ${\bf X}^V$ is recurrent and ${\rm Sect}_N\leq\kappa$ with $\kappa>0$.
Let $u:M\to N$ be a $V$-harmonic map whose range is contained in a regular geodesic ball $B_R(o)$.
Then $u$ is a constant map.
\end{thm}

\begin{remark}\label{rem:Liouville3} 
{\rm We give the following remark.

\begin{enumerate}

\item As far as we know, Theorem \ref{thm:Liouville1}, Theorem~\ref{thm:Liouville2} and Theorem ~\ref{thm:Liouville3}  are new in the literature, even in the symmetric case 
$V=\nabla f$ for some $f\in C^2(M)$. 

\item Under ${\rm Ric}_V^m\geq0$, the condition {\bf (A3)} implies {\bf (B3)} (see 
Theorem~\ref{thm:(A)0}). 
Moreover, the statement of Theorem~\ref{thm:Liouville2} remains valid by replacing {\bf (A3)} 
with {\bf (B3)}, because Theorem~\ref{thm:Liouville2} is  based on Theorem~\ref{thm:Liouville1}. 
\item The assertion of Theorem~\ref{thm:Liouville3} under $V=0$ is well-known. Indeed, if
the Brownian motion ${\bf X}^V$ is recurrent, then every bounded harmonic function is constant. So the bounded Liouville property into regular geodesic balls of
Riemannian manifolds having  positive upper sectional curvature bounds holds by
Kendall~\cite{Kend:martingalemanifold, Kend:probconvI, Kend:hemisphere}.
\item Theorem~\ref{thm:Liouville3} is an extension of \cite[Theorem~5]{ChenJostQiu}, which states that  if $(M, g)$ is a compact Riemannian manifold without boundary and $(N, h)$
is a complete Riemannian manifold with sectional curvature bounded above by a positive
constant $\kappa$, then every $V$-harmonic map $u: M\rightarrow N$ with range in a regular ball  must be a constant map.
We would like to point out that the requirement of the compactness of $M$ is indeed for the validity of the following strong maximum principle: If a smooth $\Delta_V$-harmonic function $u$ attains a maximum $u(x_0)$ at $x_0\in M$, then
$u\equiv u(x_0)$. This is nothing but the bounded Liouville property for $\Delta_V$-harmonic functions. The proof of  Theorem~\ref{thm:Liouville2} shows  that the
  bounded Liouville property holds for $\Delta_V$-harmonic functions on complete Riemannian manifolds with recurrent $\Delta_V$-diffusion process. In view of this, Theorem~\ref{thm:Liouville3} gives 
  the connection between the bounded Liouville property for 
$V$-harmonic maps with range in regular geodesic ball
of complete Riemannian manifolds with positive upper sectional curvature bounds and the recurrence property of $\Delta_V$-diffusion process  on the complete  source 
 Riemannian manifolds.

\end{enumerate}
}
\end{remark}

As a corollary of Theorem~\ref{thm:Liouville3}, the recurrence of $\Delta_f$-diffusion process on complete Riemannian surface proved in \cite[Theorem~2.6]{KL}
under the lower boundedness of $f$ and ${\Ric}_f^1\geq0$
 leads us to the following:
\begin{cor}\label{cor:Liouville3}
Let $(M,g)$ be a \,complete  Riemannian surface with a lower bounded function $f\in C^2(M)$, and $(N,h)$ a complete smooth Riemannian manifold. Assume ${\Ric}_f^1\geq0$ and ${\rm Sect}_N\leq\kappa$ with $\kappa>0$.
Let $u:M\to N$ be an $f$-harmonic map whose range is contained in a regular geodesic ball $B_R(o)$.
Then $u$ is a constant map.
\end{cor}

\begin{example}\label{ex:Liouville1}
{\rm Let $M=\R^n$ be an $n$-dimensional Euclidean space. Consider a smooth lower bounded function 
$f(x):=(n-m)\log(n-m+|x|^2)$ for $m\in\,]-\infty,\,1\,]$. Then we see that for $x,v\in\R^n$,
\begin{align*}
{\Ric}_f^{\infty}(v,v)(x)&=\frac{2(n-m)}{(n-m+|x|^2)^2}\left((n-m+|x|^2)|v|^2-\langle x,v\rangle^2 
\right),\\
{\Ric}_f^m(v,v)(x)&=\frac{2(n-m)}{(n-m+|x|^2)^2}\left((n-m+|x|^2)|v|^2-\langle x,v\rangle^2\right)\geq0.
\end{align*}
When $f(x):=n\log(n+|x|^2)$, then ${\Ric}_f^0(v,v)(x)\geq0$, hence 
one can apply Theorem~\ref{thm:Liouville1} for any $f$-harmonic maps having {\bf(G1)} condition into Hadamard manifolds. 
In this case, the condition $\text{\bf (A1)}^{*}$ below holds, because $f$ is radially increasing and 
rotationally symmetric around $0$. 
Since $\text{\bf (A1)}^{*}$ implies $\text{\bf (A3)}^{*}$, hence ${\bf (B3)}$ holds (see the prof of Theorem~\ref{thm:(A)0} below), Theorem~\ref{thm:Liouville2} can be applied. 
}
\end{example}

\begin{example}\label{ex:Liouville3}
{\rm Let $M=\R^n$ be an $n$-dimensional Euclidean space. Consider a smooth lower bounded function $f(x):=(n-1)\log(n-1+|x|^2)$.
Then we see that for $x,v\in\R^n$
\begin{align*}
{\Ric}_f^1(v,v)(x)&=\frac{2(n-1)}{(n-1+|x|^2)^2}\left((n-1+|x|^2)|v|^2-\langle x,v\rangle^2 \right)\geq0.
\end{align*}
Then by \cite[Theorem~2.6]{KL}, the $\Delta_f$-diffusion process is recurrent
under the lower boundedness of $f$, $n=2$ and ${\Ric}_f^1\geq0$. In this case, any $f$-harmonic map into regular geodesic ball in a complete
Riemannian manifold $(N,h)$ with ${\rm Sect}_N\leq \kappa$ ($\kappa>0$) is a constant map by 
Theorem~\ref{thm:Liouville3}.
}
\end{example}

When $N=\R$, Theorem~\ref{thm:Liouville1} extends Theorem~\ref{thm:Brighton}, because of the order ${\Ric}_V^m\geq{\Ric}_V^{\infty}$ for $m\in]-\infty,\,0\,]$. Moreover, it extends the classical S.~Y.~Cheng's Liouville Theorem~\cite{SYCheng}  for sublinear growth harmonic maps between Riemannian manifolds. Indeed,  based on the Laplacian comparison theorem due to Bakry and Qian \cite{BQ}, S.~Y.~Cheng's Liouville Theorem for sublinear growth $f$-harmonic functions has been extended in a non-submitted paper in the 2007 Habilitation Thesis of the third named author 
to symmetric diffusion operator $\Delta_f$ on complete
Riemannian manifolds with $\Ric_f^m\geq 0$ for $m\in[\,n,\,+\infty\,[$. The content of \cite{SYCheng} is extended to the setting of metric measure space (see Hua-Kell-Xia \cite{HuaKellXia} and Kuwada-Kuwae~\cite{KK}) for sublinear growth harmonic functions. But our result is not covered by them.
Theorem~\ref{thm:Liouville2} also extends the classical results on
the bounded Liouville property for harmonic maps from complete smooth Riemannian manifolds with non-negative Ricci curvature into regular geodesic balls of
complete smooth Riemannian manifolds with positive upper sectional curvature bounds
proved by Hildebrandt-Jost-Wideman~\cite{HildtJostWidemann}
and Choi~\cite{Choi}. 
Moreover, Theorems~\ref{thm:Liouville1} and \ref{thm:Liouville2}  extend the bounded Liouville properties for $V$-harmonic maps proved in
\cite[Theorem~2]{ChenJostQiu,Qiu}
under ${\rm Ric}_V^{\infty}\geq0$. Note that there was no such Liouville theorem for sublinear growth $V$-harmonic maps into Cartan-Hadamard manifolds in the past literature.
In Theorem~\ref{thm:Liouville1}, we do not assume the boundedness condition for the $V$-harmonic map 
as in \cite[Theorem~2]{ChenJostQiu,Qiu}
when $(N,h)$ is a Cartan-Hadamard manifold. As in Remark~\ref{rem:Liouville3}(iii), Theorem~\ref{thm:Liouville3} is an extension of  \cite[Theorem~5]{ChenJostQiu}. Note that Corollary~\ref{cor:Liouville3} can not be deduced from
Theorem~\ref{thm:Liouville2}. 

Our proof of Theorem~\ref{thm:Liouville1} based on the method by \cite{Staff:Liouville} is stochastic, but the proof in
\cite[\S 2 Lower bound]{Staff:Liouville} is questionable, 
because in \cite{Staff:Liouville} there is no explicit proof for the sublinear 
growth condition for the norm $|\d u|$ of 
differential map $\d u:TM\to TN$ for harmonic map $u:M\to N$,  
and
the method of \cite[\S3 Upper bound]{Staff:Liouville} does not work well by using the Laplacian comparison theorem
under ${\Ric}_V^m\geq0$ with $m\in ]-\infty,\,1\,]$ (see Section~\ref{sec:(A)} 
for such Laplacian comparison theorems). 
To overcome this difficulty,  
in addition to the non-negative bound for $(m,V)$-Ricci curvature, 
we consider conditions {\bf (A{\boldmath$i$})} for $V$-Laplacian acting on the radial function (see Section~\ref{sec:(A)}). 
Moreover, we prove that for each $i=1,2,3$ the growth condition {\bf (G{\boldmath$i$})} for the $V$-harmonic map $u$ implies the same growth condition {\bf (G{\boldmath$i$})} for $|\d u|$ under {\bf (B3)} (see Lemma~\ref{lem:lineargrowth}).  
This fills the incompleteness of the arguments in
\cite[\S2 Lower bound]{Staff:Liouville}. The proof of
Theorem~\ref{thm:Liouville2} is based on Theorem~\ref{thm:Liouville1} and
the inheritance of the bounded Liouville property for functions to maps firstly
shown by W.~Kendall (see \cite{Kend:martingalemanifold, Kend:probconvI, Kend:hemisphere}). Finally, we emphasize that our results obtained by stochastic method are 
new, that is, there were no such results obtained by geometric analysis in the past literature.

\section{Laplacian comparison theorem}\label{sec:(A)}
Let $V$ be a $C^1$-vector field on a complete Riemannian manifold $(M,g)$. 
In general, $V$ is not necessarily a gradient type like  
$V=\nabla f$. However, we can make a function $f_V$ which plays a similar role like $f$ provided 
$V=\nabla f$.
Define 
\begin{align*}
V_{\gamma}(r):&=\int_0^r\langle V_{\gamma_s},\dot{\gamma}_s\rangle\d s
\end{align*}
for a unit speed geodesic $\gamma:[\,0,\,+\infty\,[\to M$, 
and 
\begin{align}
f_V(x):&=\inf\left\{\left.
\int_0^{r_p(x)}\langle V_{\gamma_s},\dot{\gamma}_s\rangle \d s
\;\right| \left.\begin{array}{ll}&\gamma: \text{unit speed geodesic}\\ & \gamma_0=p, \gamma_{r_p(x)}=x\end{array}\right.\right\}.\label{eq:ModifiedPhi}
\end{align} 
Note that $V_{\gamma}$ depends on the choice of unit speed geodesic $\gamma$, and 
$f_V(x)$ depends on $p$ with $f_V(p)=0$ and it is well-defined for $x\in M$. 
It is easy to see that $f_V(x)=\int_0^{r_p(x)}Vr_p(\gamma_s)\d s$ for $x\notin {\rm Cut}(p)$, where 
$\gamma$ is the unique unit speed geodesic with $\gamma_0=p$ and $\gamma_{r_p(x)}=x$. 
Moreover, for $x\notin {\rm Cut}(p)$, $f_V(x)=V_{\gamma}(r_p(x))$ for the unique unit speed geodesic $\gamma$ with $\gamma_0=p$ and $\gamma_{r_p(x)}=x$. Hence $f_V(\gamma_t)=V_{\gamma}(t)$ for any unit speed geodesic $\gamma$ with $\gamma_0=p$ and $\gamma_t\notin {\rm Cut}(p)$.
When $V=\nabla f$ is a gradient vector field for some $ f\in C^2(M)$, then one can see 
\begin{align*}
V_{\gamma}(t)&=\int_0^t\langle \nabla f,\dot{\gamma}_s\rangle\d s
=\int_0^t\frac{\d}{\d s}f(\gamma_s)\d s=f(\gamma_t)-f(\gamma_0).
\end{align*}
Throughout this section, 
we fix a point $p\in M$ and a constant $C=C_p>0$, which may depend on $p$. 
For $x\in M$, we define
\begin{align*}
s_p(x):=\inf\left\{\left. C_p 
\int_0^{r_p(x)}e^{-\frac{2V_{\gamma}(t)}{n-m}}\d t\;\right|\left.\begin{array}{ll}&\gamma: \text{unit speed geodesic}\\ & \gamma_0=p, \gamma_{r_p(x)}=x\end{array}\right.\right\}.
\end{align*}
When $V=\nabla f$ with $f\in C^2(M)$ and $C_p=\exp\left(-\frac{2f(p)}{n-m} \right)$, then 
\begin{align*}
s_p(x)=\inf\left\{\left.  
\int_0^{r_p(x)}e^{-\frac{2f(\gamma_t)}{n-m}}\d t\;\right|\left.\begin{array}{ll}&\gamma: \text{unit speed geodesic}\\ & \gamma_0=p, \gamma_{r_p(x)}=x\end{array}\right.\right\}.
\end{align*}

Since $(M,g)$ is complete, $s_p(x)$ is finite  and 
well-defined from the basic properties of Riemannian geodesics.

\medskip

The following theorem is a special case of \cite[Theorem~2.5]
{KS} (see also \cite[Theorem~2.3]{KSa}).

\begin{thm}[Laplacian Comparison Theorem]\label{thm:GlobalLapComp}
Fix a point $p\in M$, $\kappa\in\R$ and $m\in]-\infty, \,1\,]$. 
Suppose 
\begin{align}
{\Ric}_V^m\geq (n-m)\kappa e^{-\frac{4f_V(x)}{n-m}} C_p^2 g\label{eq:RicciLowerf_V}
\end{align} 
for $x\notin {\rm Cut}(p)\cup\{p\}$. 
Then  
\begin{align}
(\Delta_V r_p)(x)\leq (n-m)\cot_{\kappa}(s_p(x))
e^{-\frac{2f_V(x)}{n-m}}  C_p \label{eq:GloLapComp}
\end{align}
holds for such $x\notin({\rm Cut}(p)\cup \{p\})$. 
When $V=\nabla f$ for some $f\in C^2(M)$ and $C_p=\exp\left(-\frac{2f(p)}{n-m}\right)$, then 
\begin{align}
{\Ric}_V^m\geq (n-m)\kappa e^{-\frac{4f(x)}{n-m}}g\label{eq:RicciLowerf}
\end{align} 
for $m\in]-\infty,\,1\,]$
yields that 
\begin{align}
(\Delta_V r_p)(x)\leq (n-m)\cot_{\kappa}(s_p(x))
e^{-\frac{2f(x)}{n-m}} \label{eq:GloLapCompGradf}
\end{align}
holds for $x\notin({\rm Cut}(p)\cup \{p\})$. Here $\cot_{\kappa}(r)=\frac{\mathfrak{s}_{\kappa}'(r)}{\mathfrak{s}_{\kappa}(r)}$ with 
\begin{align*}
\mathfrak{s}_{\kappa}(r):=\left\{\begin{array}{cc}\frac{\sin(\sqrt{\kappa}r)}{\sqrt{\kappa}} & \kappa>0, \\ r & \kappa=0, \\\frac{\sinh(\sqrt{-\kappa}r)}{\sqrt{-\kappa}} & \kappa<0.\end{array}\right.
\end{align*}
\end{thm}
Recall that $f_V$ depends on $p$. 
We define $\overline{f}_V(r):=\sup_{B_r(p)}f_V$ and $\underline{f}_V(r):=\inf_{B_r(p)}f_V$. Then $\overline{f}_V(r_p(x))$ and $\underline{f}_V(r_p(x))$ are rotationally symmetric functions around $p$. 
We now consider the following conditions, which are weaker than the conditions 
{\bf (A1)}, {\bf (A2)} and {\bf (A3)}, respectively:
\begin{enumerate}
\item[{\bf (A1)$^{*}$}\hspace{-0.2cm}]\; $f_V$ is bounded, or $t\mapsto f_V(\gamma_t)$ is increasing for any unit speed geodesic $\gamma$ with $\gamma_0=p$, 
\item[{\bf (A2)$^{*}$}\hspace{-0.2cm}]\; 
There exists $D\geq1$ such that $\overline{f}_V(r)-\underline{f}_V(r)\leq\frac{n-m}{2}\log D(1+r)$.
\item[{\bf (A3)$^{*}$}\hspace{-0.2cm}]\; 
There exists $D\geq1$ such that $\overline{f}_V(r)-\underline{f}_V(r)\leq\frac{n-m}{2}\log D(1+r^2)$.\end{enumerate}
Indeed, in view of $|f_V(x)|\leq\int_0^{r_p(x)}v(s)\d s$, 
$f_V(\gamma_t)=\int_0^t\langle V,\dot\gamma_s\rangle_{\gamma_s}\d s$ for unit speed geodesic 
$\gamma$ joining $\gamma_0=p$ and $\gamma_{r_p(x)}=x\notin{\rm Cut}(p)$, 
and $-\int_0^rv(s)\d s\leq \underline{f}_V(r)\leq0\leq\overline{f}_V(r)\leq \int_0^rv(s)\d s$, we see that 
$\text{\bf (A1)}$, (resp.~$\text{\bf (A2)}$, $\text{\bf (A3)}$) implies $\text{\bf (A1)}^{*}$, (resp.~$\text{\bf (A2)}^{*}$, $\text{\bf (A3)}^{*}$).

\medskip

Recall the condition $\text{\bf (B\text{\boldmath$i$})}$ for each $i=1,2,3$. 

\begin{lem}\label{lem:ConditionsEquivalence}
The condition {\bf (B3)} is equivalent to the following condition 
$\text{\bf (B3)}^{*}$.
\begin{enumerate}
\item[$\text{\bf (B3)}^{*}$] There exist $a\in]\,0,\,{\color{black}{{\sf d}}}(p,{\rm Cut}(p))\,[$ and $D\geq1$ such that $\Delta_Vr_p(x)\leq D(1+r_p(x))$ for all $x\in M\setminus ({\rm Cut}(p)\cup B_a(p))$.
\end{enumerate}
\end{lem}
\begin{proof}[{\bf Proof}]
It is easy to see that {\bf (B3)} implies $\text{\bf (B3)}^{*}$ by taking any small $a\in]\,0,\,{\color{black}{{\sf d}}}(p,{\rm Cut}(p))\,[$. 
Since $\lim_{x\to p}r_p(x)\Delta_Vr_p(x)=n-1$, taking a sufficiently small $a\in]\,0,\,{\color{black}{{\sf d}}}(p,{\rm Cut}(p))\,[$, we can obtain the converse implication 
$\text{\bf (B3)}^{*}$$\Longrightarrow${\bf (B3)}.   
\end{proof}

\medskip

Theorem~\ref{thm:GlobalLapComp} tells us the following: 
\begin{thm}\label{thm:(A)0}
Suppose ${\rm Ric}_V^m\geq0$ for $m\in]-\infty,\,1\,]$. 
If $\text{\bf (A1)}$ {\rm(}resp.~$\text{\bf (A2)}$, $\text{\bf (A3)}${\rm)} holds, then the condition {\bf (B1)} 
{\rm(}resp.~{\bf (B2)}, {\bf (B3)}{\rm)} holds. 
\end{thm}
\begin{proof}[\bf Proof]
Note first that $\text{\bf (A1)}$ (resp.~$\text{\bf (A2)}$, $\text{\bf (A3)}$) implies 
$\text{\bf (A1)}^{*}$ (resp.~$\text{\bf (A2)}^{*}$, $\text{\bf (A3)}^{*}$). 
Take $x\notin {\rm Cut}(p)$. Then there exists a unique unit speed geodesic $\gamma_t$ with $\gamma_0=p$ and $\gamma_{r_p(x)}=x$ such that 
\begin{align*}
s_p(x)=C_p\int_0^{r_p(x)}\exp\left(-\frac{2f_V(\gamma_t)}{n-m}\right)\d t.
\end{align*}
It is easy to see that 
\begin{align*}
f_V(\gamma_t)-f_V(x)\leq 
\left\{\begin{array}{ll} 
2\|f_V\|_{\infty}\quad\text{ or }\quad 0 & \text{ under 
$\text{\bf (A1)}^{*}$},
\\
\overline{f}_V(r_p(x))-\underline{f}_V(r_p(x))\leq \frac{n-m}{2}\log D(1+r_p(x))
 & \text{ under $\text{\bf (A2)}^{*}$}, \\
 \overline{f}_V(r_p(x))-\underline{f}_V(r_p(x))\leq \frac{n-m}{2}\log D(1+r_p^2(x))
 & \text{ under $\text{\bf (A3)}^{*}$}.
 \end{array}\right.
\end{align*}

The comparison \eqref{eq:GloLapComp} becomes 
\begin{align*}
r_p(x)\Delta_Vr_p(x)&\leq \frac{n-m}{\frac{1}{r_p(x)}\int_0^{r_p(x)}\exp\left(\frac{2(f_V(x)-f_V(\gamma_t))}{n-m}\right)\d t}\\
&\leq \left\{\begin{array}{lc}
(n-m)e^{\frac{2\|f_V\|_{\infty}}{n-m}}\quad\text{ or }\quad n-m & \text{ under $\text{\bf (A1)}^{*}$},
\\
(n-m)D(1+r_p(x)) & \text{ under $\text{\bf (A2)}^{*}$,}\\
(n-m)D(1+r_p^2(x)) & \text{ under $\text{\bf (A3)}^{*}$} \end{array}\right.
\end{align*}
for $x\notin {\rm Cut}(p)$.
\end{proof}

\begin{thm}\label{thm:(A)-K}
Suppose ${\rm Ric}_V^m\geq- Kg$ with $K>0$ for $m\in]-\infty,\,1\,]$. 
Assume that $f_V$ is bounded. Then we have {\bf (B2)}. 
\end{thm}
\begin{proof}[\bf Proof]
Since $f_V$ is bounded, we have \eqref{eq:RicciLowerf_V} with $K=-(n-m)\kappa e^{-\frac{4\inf_Mf_V}{n-m}}C_p^2$. Then, for $x\notin {\rm Cut}(p)$, \eqref{eq:GloLapComp} implies
\begin{align*}
r_p(x)\Delta_Vr_p(x)&\leq (n-m)\sqrt{-\kappa}\,r_p(x)\coth(\sqrt{-\kappa}s_p(x))e^{-\frac{2f_V(x)}{n-m}}C_p\\
&\leq (n-m)r_p(x)\left(\sqrt{-\kappa}+\frac{1}{s_p(x)} \right)e^{-\frac{2f_V(x)}{n-m}}C_p\quad  (\because\; \coth x\leq 1+x^{-1}, \;x>0)\\
&=(n-m)\left(\sqrt{-\kappa}\,r_p(x)+\frac{r_p(x)}{s_p(x)} \right)e^{-\frac{2f_V(x)}{n-m}}C_p\\
&\leq (n-m)\left(\sqrt{-\kappa}\,e^{-\frac{2\inf_M f_V}{n-m}}C_pr_p(x)+\frac{r_p(x)}{s_p(x)}e^{-\frac{2f_V(x)}{n-m}}C_p \right)\\
&\leq (n-m)\left(\sqrt{-\kappa}\,e^{-\frac{2\inf_M f_V}{n-m}}C_pr_p(x)+
e^{\frac{2(\sup_Mf_V-\inf_M f_V)}{n-m}}
 \right).
\end{align*}
This yields the condition {\bf (B2)}. 
\end{proof}

\begin{cor}\label{cor:(A)}
Assume that $V=\nabla f$ for some $f\in C^2(M)$ and $f$ is bounded. 
Suppose ${\rm Ric}_f^m\geq- Kg$ with $K>0$ for $m\in]-\infty,\,1\,]$. 
Then the condition {\bf (B2)} holds. 
\end{cor}

The following is shown in \cite[(2.6)]{Qiu}.

\begin{thm}[{{\cite[(2.6)]{Qiu}}}]\label{thm:LaplacianCompaInfinity}
Suppose ${\rm Ric}_V^{\infty}\geq -Kg$ for $K\geq0$. Then for any sufficiently small 
$a\in]\,0,\,{\color{black}{{\sf d}}}(p,{\rm Cut}(p))\,[$
there exists $C>0$ such that 
$\Delta_V r_p(x)\leq C+Kr_p(x)$ for all $x\in M\setminus ({\rm Cut}(p)\cup B_a(p))$. 
In particular, $\text{\bf (B3)}^{*}$, hence {\bf (B3)} holds by Lemma~\ref{lem:ConditionsEquivalence}. 
\end{thm}

\begin{remark}\label{rem:LaplacianCompaInfinity}
{\rm As a special case of Theorem~\ref{thm:LaplacianCompaInfinity}, ${\rm Ric}_V^{\infty}\geq0$ implies ${\bf (B3)}$. In particular, our Theorem~\ref{thm:Liouville2} 
extends Chen-Jost-Qiu~\cite[Theorem~2]{ChenJostQiu} and Qiu~\cite[Theorem~2]{Qiu}.  
}
\end{remark}
\begin{cor}\label{cor:Vbounded}
Suppose that ${\rm Ric}_V^m\geq -Kg$ for $K\geq0$ with $m\in]-\infty,\,1\,]$ and $V$ is a bounded $C^1$-vector field. Then the condition {\bf (B3)} holds.
\end{cor}
\begin{proof}[\bf Proof]
${\rm Ric}_V^m\geq -Kg$ implies ${\rm Ric}_V^{\infty}(v,v)\geq -K|v|^2+\frac{\langle V,v\rangle^2}{m-n}\geq-
\left(K+\frac{|V|_{\infty}^2}{n-m}\right)|v|^2$, i.e., ${\rm Ric}_V^{\infty}\geq -\left(K+\frac{|V|_{\infty}^2}{n-m}\right)g$. Then the conclusion follows from 
Theorem~\ref{thm:LaplacianCompaInfinity}.
\end{proof}

\begin{remark}\label{rem:classicalLaplacian}
{\rm If ${\Ric}_V^m\geq (m-1)\kappa g$ with $\kappa<0$ for $m\in[\,n,\,+\infty\,[$, then we know the following Laplacian comparison 
\begin{align*}
r_p(x)\Delta_Vr_p(x)&\leq (m-1)\sqrt{-\kappa}\,r_p(x)\coth(\sqrt{-\kappa}\,r_p(x))\\
&\leq (m-1)(1+\sqrt{-\kappa}\,r_p(x))
\quad\text{ for }\quad x\notin {\rm Cut}(p)
\end{align*}
by \cite[Theorem~4.2]{BQ}, \cite[Theorem~1.1]{Xdli:Liouville}, \cite[Theorem~2.3]{KSa} and  \cite[Theorem~3.9]{LMO:CompaFinsler}. In particular, the condition {\bf (B2)} holds under ${\Ric}_V^m\geq Kg$ for $m\in[\,n,\,+\infty\,[$ and $K\in\R$. 
} 
\end{remark}

\section{Construction of $\Delta_V$-diffusion processes and
Kendall's expression of radial processes}\label{sec:diffusion}

In this section, we summarize the construction of
$\Delta_V$-diffusion processes on $M$ and note that the Kendall's expression for radial processes
remains valid for $\Delta_V$-diffusion process.

Let $V$ be a $C^1$-vector field and $\Delta_V:=\Delta-\langle V,\nabla\cdot\rangle$ denote the $V$-Laplacian. The construction of
$\Delta_V$-diffusion process can be done by way of the \emph{horizontal diffusion process} $(U_t)_{t\geq 0}$ generated by $\Delta_{O(M)}-{\bf H}_V$ on $O(M)$ by solving the Stratonovich stochastic differential equation (SDE)
\begin{align*}
\d U_t=\sqrt{2}\sum_{i=1}^n H_{e_i}(U_t)\circ \d B_t^i-{\bf H}_V(U_t)\d t,\quad U_0=U\in O(M),
\end{align*}
where $B_t:=(B_t^1,B_t^2,\cdots, B_t^n)$ is the $n$-dimensional Brownian motion on a complete filtered probability space
$(\Omega,(\mathscr{F}_t)_{t\geq0},\mathscr{F}, \mathbb{P})$.
Here $(e_i)_{i=1}^n$ is the canonical orthonormal basis in $\R^n$, $(H_{e_i})_{i=1}^n$ the corresponding family of horizontal vector fields in the frame bundle $\mathscr{F}(M)$, and
${\bf H}_V(U):=H_{U^{-1}V}(U)$, $U\in  O(M)$ denotes the horizontal lift of $V$, where $U^{-1}V$ is the unique vector $e\in\R^n$ such that $V_{\pi(U)}=Ue$. Moreover,
\begin{align*}
\Delta_{O(M)}:= \sum_{i=1}^n H_{e_i}^2
\end{align*}
is the  \emph{horizontal Laplace operator}.
Since
${\bf H}_V$ is $C^1$, it is  well known that the equation has a unique solution up to the life time $\zeta:=\lim_{k\to\infty}\zeta_k$, where
\begin{align*}
\zeta_k:=\inf\{t>0\mid {\color{black}{{\sf d}}}(\pi(U),\pi(U_t))\geq k\},\quad  k\in\mathbb{N}
\end{align*}
and $\pi:O(M)\to M$ is the projection from the orthonormal frame to its point on $M$.
Let $X_t:=\pi(U_t)$. Then $X_t$ solves the equation
\begin{align*}
\d X_t=\sqrt{2} U_t\circ \d B_t-V(X_t)\d t, \quad
X_0=x:=\pi(u_0)
\end{align*}
up to the life time $\zeta$. By It\^o's formula, for any $f\in C_0^2(M)$,
\begin{align}
m_t^f:=f(X_t)-f(X_0)-\int_0^t \Delta_Vf(X_s)\d s=\sqrt{2}\int_0^t\langle U_s^{-1}\nabla f(X_s),\d B_s\rangle\label{eq:Ito}
\end{align}
is a martingale; i.e., $X_t$ is the \emph{diffusion process} generated by $\Delta_V$, and we call it the \emph{$\Delta_V$-diffusion process}. 
The expression \eqref{eq:Ito} can be extended for $f\in C^{\infty}(M)$. 
In this case, $m_t^f$ is no longer a martingale, but for any relatively compact open set $G$, $t\mapsto m^f_{t\land \tau_G}$ 
is a martingale. Indeed, there exists $\rho_G\in C_0^{\infty}(M)$ such that $\rho_G\equiv 1$ on $G$, and $f_G\in C_0^{\infty}(M)$ defined by $f_G:=f\rho_G$ satisfies that $m^f_{t\land\tau_G}=m^{f_G}_{t\land\tau_G}$ is a martingale. 

From now on, the probability measure $\mathbb{P}$
is taken for the underlying process starting from point $x$, i.e. $\mathbb{P}:=\mathbb{P}_x$.
The details of the construction for $\Delta_V$-diffusion process can be found in \cite[\S 2.1]{FYWang}. We denote the $\Delta_V$-diffusion process
by ${\bf X}^V:=(\Omega,X_t,(\mathscr{F}_t)_{t\geq0},\mathscr{F},{\P}_x)$.

The following is shown in \cite[Proposition~1.1.5]{Hsu:2001}.

\begin{thm}[{\cite[Proposition~1.1.5]{Hsu:2001}}]\label{thm:tight}
The diffusion process ${\bf X}^V$
on $M$ is tight in the sense that for any increasing sequence $\{G_n\}$ of relatively compact open sets  satisfying $\overline{G_n}\subset G_{n+1}\subset M$ for any $n\in\N$ and $M=\bigcup_{n=1}^{\infty}G_n$, 
\begin{align}
{\P}_x\left(\lim_{n\to\infty}\tau_{G_n}=\zeta\right)=1, \label{eq:nest}
\end{align}
{\color{black}{equivalently
\begin{align}
{\P}_x\left(\lim_{n\to\infty}\sigma_{M\setminus G_n}\geq \zeta\right)=1 \label{eq:nest1}
\end{align}
for all $x\in M$. Note here that $\tau_{G_n}=\sigma_{M\setminus G_n}\land\zeta$ and $\sigma_{M\setminus G_n}\geq\zeta$ is equivalent to $\sigma_{M\setminus G_n}=+\infty$. 
}} 
\end{thm} 
{\color{black}{
\begin{lem}\label{lem:noinsidekilling}
Let $\{G_n\}$ be an increasing sequence $\{G_n\}$ of relatively compact open sets  satisfying $\overline{G_n}\subset G_{n+1}\subset M$ for any $n\in\N$ and $M=\bigcup_{n=1}^{\infty}G_n$. 
The following assertions are equivalent and hold. 
\begin{enumerate} 
\item\label{item:1} The diffusion process ${\bf X}^V$
on $M$ has no killing inside, i.e., 
\begin{align}
\P_x(X_{\zeta-}\in M,\zeta<\infty)=0\label{eq:insidekilling}
\end{align}
for all $x\in M$. 
\item\label{item:2} We have 
\begin{align}
\P_x(\sigma_{M\setminus G_n}<\zeta \quad\text{ for all }\quad n\in\mathbb{N})=1\label{eq:insidebehavior}
\end{align} 
for all $x\in M$. 
\item\label{item:3} We have  
\begin{align}
{\P}_x\left(\lim_{n\to\infty}\sigma_{M\setminus G_n}= \zeta\right)=1 \label{eq:nest2}
\end{align}
for all $x\in M$. 
\end{enumerate}
\end{lem}
\begin{proof}[\bf Proof]
We first prove \ref{item:1}. 
Let $\{G_n\}$ be the exhaustion as stated above. 
Let us consider a pre-semi-Dirichlet form $(\mathscr{E}_{G_n},C_0^1(G_n))$ on $L^2(G_n;\nu_g)$ defined by 
\begin{align}
\mathscr{E}_{G_n}(f,g)=\int_{G_n}\langle \nabla f,\nabla g\rangle\d\nu_g+\int_{G_n}\langle V,\nabla f\rangle \,g\,\d\nu_g\quad \text{ for }\quad f,g\in C_0^1(G_n).\label{eq:presemiDir}
\end{align}
Thanks to the boundedness $V$ on $G_n$, there exists $\alpha_0=\alpha_0(G_n)$ such that  
\begin{align}
(\mathscr{E}_{G_n})_{\alpha_0}(u,u)\geq0\quad \text{ for }\quad u\in C_0^1(G_n),\label{eq:lowerBdd}
\end{align}
where $(\mathscr{E}_{G_n})_{\alpha_0}(u,u):=\mathscr{E}_{G_n}(u,u)+\alpha_0\int_{G_n}|u|^2\d\nu_g$,  
and there exists $K=K(G_n)\geq1$ such that
\begin{align}
|\mathscr{E}_{G_n}(u,v)|\leq K(\mathscr{E}_{G_n})_{\alpha_0}(u,u)^{\frac12}
(\mathscr{E}_{G_n})_{\alpha_0}(v,v)^{\frac12}\quad\text{ for all }\quad u,v\in C_0^1(G_n).
\label{eq:Presector}
\end{align}
Moreover, for any $\alpha>\alpha_0$, there exists $0<\lambda_1\leq\lambda_2$ depending on $\alpha$ such that 
\begin{align}
\lambda_1{\bf D}_1(u,u)\leq (\mathscr{E}_{G_n})_{\alpha}(u,u)\leq\lambda_2{\bf D}_1(u,u)\quad\text{ for all }\quad u\in  C_0^1(G_n),\label{eq:EquivalenceEnergy}
\end{align}
where ${\bf D}_1(u,u):=\int_{G_n}|\nabla u|^2\d\nu_g+\int_{G_n}|u|^2\d\nu_g$. 
From \eqref{eq:EquivalenceEnergy}, $(\mathscr{E}_{G_n},C_0^1(G_n))$ is closable on $L^2(G_n;\nu_g)$ and denote its closure by $(\mathscr{E}_{G_n},D(\mathscr{E}_{G_n}))$ and it is a regular semi-Dirichleet form on 
$L^2(G_n;\nu_g)$ (cf. \cite[Theorems~1.5.2 and 1.5.3]{Oshima} in the Euclidean setting).  
The part process ${\bf X}_{G_n}^V$ defined by 
$X_t^{G_n}:=X_t$ if $t<\tau_{G_n}$, $X_t^{G_n}:=\partial$ if $t\geq\tau_{G_n}$, is a diffusion process 
associated with $(\mathscr{E}_{G_n},D(\mathscr{E}_{G_n}))$ on $L^2(G_n;\nu_g)$. 
Thanks to \cite[Theorem~3.5.16]{Oshima}, ${\bf X}_{G_n}^V$ has no killing inside in the sense that  
\begin{align}
\P_x(X_{\tau_{G_n}-}\in G_n,\tau_{G_n}<+\infty)=0\quad \nu_g\text{-a.e.~}x\in G_n.\label{eq:noinsidekillingG_n}
\end{align}
From \eqref{eq:noinsidekillingG_n}, we can deduce
\begin{align}
\P_x(X_{\zeta-}\in M,\zeta<+\infty)=0\quad \nu_g\text{-a.e.~}x\in M.\label{eq:noinsidekillingM}
\end{align}
Indeed, under $\zeta<+\infty$, we have $\lim_{n\to\infty}X_{\tau_{G_n}-}=X_{\zeta-}$ immediately from \eqref{eq:nest}. 
Take $n>\ell>k$. Then 
\begin{align*}
\P_x(X_{\tau_{G_n}-}\in G_{\ell},\zeta<+\infty)&\leq \P_x(X_{\tau_{G_n}-}\in G_{\ell},\tau_{G_n}<+\infty)\\
&\leq \P_x(X_{\tau_{G_n}-}\in G_n,\tau_{G_n}<+\infty)=0\quad \text{ for }\quad \nu_g\text{-a.e.~}x\in G_k.
\end{align*}
Thus 
\begin{align*}
\E_x[\1_{G_{\ell}}(X_{\tau_{G_n}-}):\zeta<+\infty]=0\quad\nu_g\text{-a.e.~}x\in G_k.
\end{align*}
From this, 
\begin{align*}
\E_x[\1_{G_{\ell}}(X_{\zeta-}):\zeta<+\infty]&\leq \E_x[\varliminf_{n\to\infty}\1_{G_{\ell}}(X_{\tau_{G_n}-}):\zeta<+\infty]\\
&\leq 
\varliminf_{n\to\infty}\E_x[\1_{G_{\ell}}(X_{\tau_{G_n}-}):\zeta<+\infty]=0\quad\nu_g\text{-a.e.~}x\in G_k.
\end{align*}
Letting $\ell\to\infty$ and $k\to\infty$, we obtain \eqref{eq:noinsidekillingM}. 
Next we prove that $\P_x(X_{\zeta-}\in M,\zeta<+\infty)$ is an excessive function with respect to ${\bf X}^V$. 
Indeed, by $\zeta\circ\theta_t=(\zeta-t)_+$, 
\begin{align*}
p_t(\P_{\text{\tiny$\bullet$}}(X_{\zeta-}&\in M,\zeta<+\infty))(x)\\&=\E_x[\P_{X_t}(X_{\zeta-}\in M,\zeta<+\infty)]\\
&=\P_x\left(X_{(t+\zeta\circ \theta_t)-}\in M, t+\zeta\circ \theta_t<+\infty\right)\\
&=\P_x(X_{\zeta-}\in M,t+\zeta\circ \theta_t<+\infty, t<\zeta)+
\P_x(X_{t-}\in M,t+\zeta\circ \theta_t<+\infty, t\geq\zeta)\\
&=\P_x(X_{\zeta-}\in M, t\leq\zeta<+\infty)\\
&\nearrow \P_x(X_{\zeta-}\in M, 0<\zeta<+\infty)=\P_x(X_{\zeta-}\in M, \zeta<+\infty)
\quad\text{ as }\quad t\searrow 0.
\end{align*}
On the other hand, ${\bf X}^V$ satisfies the Meyer's hypotheses (L), i.e., 
$R_{\alpha}(x,\d y)\ll \nu_g(\d y)$ for each $\alpha>0$ and $x\in M$, where 
$R_{\alpha}f(x)=\E_x\left[\int_0^{\infty}e^{-\alpha t}f(X_t)\d t\right]$ for $f\in \mathscr{B}(M)_+$. 
Indeed, since ${\bf X}_{G_n}^V$ is associated to $(\mathscr{E}_{G_n},D(\mathscr{E}_{G_n}))$, 
for  any $f\in \mathscr{B}(M)_+$ with $f=0$ $\nu_g$-a.e.
\begin{align*}
\lambda_1\| R_{\alpha}^{G_n}f\|_{L^2(G_n;\nu_g)}^2&\leq 
\lambda_1{\bf D}_1(R_{\alpha}^{G_n}f,R_{\alpha}^{G_n}f)\\
&\leq (\mathscr{E}_{G_n})_{\alpha}(R_{\alpha}^{G_n}f,R_{\alpha}^{G_n}f)\\
&=(f,R_{\alpha}^{G_n}f)_{L^2(G_n;\nu_g)}=0
\end{align*}
implies $R_{\alpha}^{G_n}f=0$ $\nu_g$-a.e. Here 
$R_{\alpha}^{G_n}f(x):=\E_x\left[\int_0^{\tau_{G_n}}e^{-\alpha t}f(X_t)\d t\right]$.  
Letting $n\to\infty$, 
we get $R_{\alpha}f=0$ $\nu_g$-a.e.~on $M$. Hence $R_{\alpha}f=0$ on $M$ by the strong Feller property of 
${\bf X}^V$ (see \cite{Molchanov},\cite[Proposition~1.11]{Azencott:Behavi}). 
Thus we obtain Meyer's hypotheses (L) for ${\bf X}^V$. 
Then, for any excessive function $h$ on $M$, 
$h=0$ $\nu_g$-a.e. implies $h=0$ on $M$ by \cite[(10.25)]{Shar:gene}. Therefore, applying this to the excessive function $h:=\P_{\text{\tiny$\bullet$}}(X_{\zeta-}\in M;\zeta<\infty)$, we obtain \ref{item:1} from \eqref{eq:noinsidekillingM}. 

Next we prove \ref{item:1}$\Longrightarrow$\ref{item:2}.  By \eqref{eq:nest1}, 
$\lim_{n\to\infty}\sigma_{M\setminus G_n}>\zeta$ or 
$\lim_{n\to\infty}\sigma_{M\setminus G_n}=\zeta$ $\P_x$-a.s.~on  $\{\zeta<+\infty\}$. 
This implies that  
$\sigma_{M\setminus G_n}=+\infty$ for some $n\in\mathbb{N}$ or 
$\sigma_{M\setminus G_n}<\zeta$ for all $n\in\mathbb{N}$ $\P_x$-a.s.~on  $\{\zeta<+\infty\}$, because 
$\sigma_{M\setminus G_n}\geq\zeta$ is equivalent to $\sigma_{M\setminus G_n}=+\infty$. 
Thanks to \eqref{eq:insidekilling} for all $x\in M$,
under $\zeta<+\infty$, the former case, i.e., $\sigma_{M\setminus G_n}=+\infty$ for some $n\in\mathbb{N}$  does not occur. Under $\zeta=+\infty$, by \eqref{eq:nest1}, $\lim_{n\to\infty}\sigma_{M\setminus G_n}=+\infty$. Applying \eqref{eq:insidekilling} for all $x\in M$ again, $\sigma_{M\setminus G_n}=+\infty$ for some $n\in\mathbb{N}$  does not occur in this case.
Therefore, we obtain \ref{item:2}. The implication \ref{item:2}$\Longrightarrow$\ref{item:3} is clear from 
\eqref{eq:nest1}. 

Finally, we prove \ref{item:3}$\Longrightarrow$\ref{item:1}. 
Suppose $\lim_{n\to\infty}\sigma_{M\setminus G_n}=\zeta$. If $\sigma_{M\setminus G_n}=\zeta$ for some $n\in\mathbb{N}$, then $\sigma_{M\setminus G_n}=\zeta=+\infty$. So under $\zeta<+\infty$, this does not occur. Hence $\sigma_{M\setminus G_n}<\zeta$ for all $n\in\mathbb{N}$ under $\zeta<+\infty$. 
Then, we see $\lim_{n\to\infty}X_{\sigma_{M\setminus G_n}}=X_{\zeta-}$ and $X_{\sigma_{M\setminus G_n}}\in M\setminus G_n$ $\P_x$-a.s.~on $\{\zeta<+\infty\}$, hence
\begin{align*}
r_p(X_{\zeta-})=\lim_{n\to\infty}r_p(X_{\sigma_{M\setminus G_n}})\geq\lim_{n\to\infty}{\color{black}{{\sf d}}}(p,G_n^c)=+\infty
\end{align*} 
holds $\P_x$-a.s.~on  $\{\zeta<+\infty\}$.  
This implies \eqref{eq:insidekilling} for all $x\in M$.
\end{proof}
}}
A function $f$ on $M$ is said to be \emph{$V$-subharmonic} if it belongs to $C^2(M)$ and $\Delta_Vf\geq0$ on $M$. A function $f$ on $M$ is said to be \emph{$V$-superharmonic} if $-f$ is $V$-subharmonic. A function $f$ on $M$ is said to be 
\emph{$V$-harmonic} if it is $V$-subharmonic and $V$-superharmonic.
A function $f$ on $M$ is said to be \emph{finely $V$-subharmonic} if it is finely continuous (see \cite{BG:Markov} for the definition of fine topology)
and $t\mapsto f(X_t)$ is a local (${\P}_x$-)submartingale on $[\,0,\,\zeta\,[$ for all $x\in M$, i.e., there exists an increasing sequence of stopping times $\{T_k\}$ with $T_k<\zeta$ for all $k\in \N$ and $\lim_{k\to\infty}T_k=\zeta$ such that $t\mapsto f(X_{t\land T_k})$ is a submartingale.  A function $f$ on $M$ is said to be \emph{finely $V$-superharmonic} if $-f$ is finely $V$-subharmonic. A function $f$ on $M$ is said to be \emph{finely $V$-harmonic} if it is finely $V$-subharmonic and finely $V$-superhamonic. 
If $f\in C^2(M)$, then the
both sides of \eqref{eq:Ito} are local (${\P}_x$-)martingales on $[\,0,\,\zeta\,[$. More precisely,
let $\{G_k\}$ be the increasing sequence of relatively compact open sets.  
Then there exists $\rho_k\in C_0^{\infty}(M)$ such  that
$\rho_k\equiv1$ on ${G}_k$. We set $f_k:=f\rho_k\in C_0^2(M)$.
Then $M_t^{f_k}$ is a martingale and $M_{t\land \tau_{G_k}}^f=M_{t\land \tau_{G_k}}^{f_{k}}$, hence $M_t^f=M_t^{f_{k}}$ for $t<\tau_{G_k}$.  This means that  $M_t^f$ is a local (${\P}_x$-)martingale on $[\,0,\,\zeta\,[$ by Theorem~\ref{thm:tight}, because $\tau_{G_k}<\zeta$.
In particular, every $V$-subharmonic (resp.~$V$-superharmonic, $V$-harmonic) functions are finely $V$-subharmonic (resp.~$V$-superharmonic, $V$-harmonic).

The following theorem can be proved in a similar way with
the case of Brownian motion on $M$.
\begin{thm}[{cf.~\cite[Theorem~3.5.1]{Hsu:2001}, \cite[proof of Corollary~3.6]{BRW:invariant}, \cite[(3.7.10)]{FYWang}}]\label{thm:KendallDecom}
Suppose that $(X_t)_{t\geq0}$ is a $\Delta_V$-diffusion process on $M$.
Then there exists a one-dimensional Euclidean Brownian motion $\beta$ and a non-degenerated continuous process $L_t$ which increases only when $X_t\in {\rm Cut}(p)$ 
such that
\begin{align}
r_p(X_t)-r_p(X_0)=\sqrt{2}\beta_t+\int_0^t\Delta_Vr_p(X_s)\1_{\{X_s\notin {\rm Cut}(p)\}}\d s-L_t,\qquad t<\zeta\label{eq:Kendall}
\end{align}
${\P}_x$-a.s.~for all $x\in M$. Here $\zeta:=\inf\{t>0\mid X_t\notin M\}$ is the life time of $(X_t)_{t\geq0}$.
\end{thm}

{\color{black}{

\begin{lem}\label{lem:squaredistance}
Suppose {\bf (B3)}.
Let $\{G_n\}$ be an increasing sequence $\{G_n\}$ of relatively compact open sets  satisfying $\overline{G_n}\subset G_{n+1}\subset M$ for any $n\in\N$ and $M=\bigcup_{n=1}^{\infty}G_n$. 
\begin{align}
{\E}_x\left[r_p^2(X_{t\land\sigma_{M\setminus G_n}})\right]
&\leq \mathfrak{D}(t),
\label{eq:quadratic}
\\
{\E}_x\left[r_p^4(X_{t\land\sigma_{M\setminus G_n}})\right]
&\leq r_p(x)^4+4(3+D)\int_0^t {\mathfrak D}(s)\d s\notag \\
&\hspace{1cm}+4De^{4Dt}\int_0^t\left(r_p^4(x)+4(D+3)\int_0^s{\mathfrak D}(u)\d u\right)e^{-4Ds}\d s\label{eq:quaternic}
\end{align}
hold for all $x\in M$. Here 
$$
{\mathfrak D}(t):=r_p(x)^2+2(1+D)t+2De^{2Dt}\int_0^t(r_p^2(x)+2(1+D)s)e^{-2Ds}\d s.
$$ 
\end{lem}
\begin{proof}[\bf Proof]
By Theorem~\ref{thm:KendallDecom}, 
the radial process $r_p(X_t)$ is a semimartingale on $[0,\zeta[$ in the sense that 
$t\mapsto r_p(X_{t\land \sigma_{M\setminus G_n}})$ is a semimartingale for each $n\in\mathbb{N}$.
By way of It\^o's formula, we have
\begin{align*}
r_p^2(X_t)-r_p^2(X_0)&=2\sqrt{2}\int_0^tr_p(X_s)\d\beta_s+
2t\\
&\hspace{1.5cm}+2\int_0^t r_p(X_s)\Delta_Vr_p(X_s)\1_{M\setminus {\rm Cut}(p)}(X_s)\d s-2\int_0^t r_p(X_s)\d L_s,\\
r_p^4(X_t)-r_p^4(X_0)&=4\sqrt{2}\int_0^tr_p^3(X_s)\d\beta_s+12\int_0^t r_p^2(X_s)\d s
\\
&\hspace{1.5cm}+4\int_0^t r_p^3(X_s)\Delta_Vr_p(X_s)\1_{M\setminus {\rm Cut}(p)}(X_s)\d s-4\int_0^t r_p^3(X_s)\d L_s,
\quad 
\end{align*}
$t\in[\,0,\,\zeta\,[$, 
${\P}_x$-a.s.~for $x\in M$.
From these equalities, we can obtain \eqref{eq:quadratic} and \eqref{eq:quaternic}.  
Indeed, by \eqref{eq:insidebehavior}, 
we see from {\bf (B3)} that 
\begin{align*}
{\E}_x\left[r_p^2(X_{t\land\sigma_{M\setminus G_n}})\right]&\leq r_p^2(x)+2t+2D{\E}_x\left[\int_0^{t\land\sigma_{M\setminus G_n}}(1+r_p^2(X_s))\d s\right]\\
&=r_p^2(x)+2(1+D)t+2D{\E}_x\left[\int_0^tr_p^2(X_s)\1_{\{s<\sigma_{M\setminus G_n}\}}\d s\right]\\
&\leq r_p^2(x)+2(1+D)t+2D\int_0^t{\E}_x\left[r_p^2(X_{s\land\sigma_{M\setminus G_n}}) \right]\d s.
\end{align*}
Applying the Gronwall's inequality to  the continuous function 
\begin{align*}
t\mapsto f(t):={\E}_x\left[r_p^2(X_{t\land\sigma_{M\setminus G_n}})\right]<+\infty
\end{align*}
under \eqref{eq:insidebehavior},  
we obtain 
\eqref{eq:quadratic}. For \eqref{eq:quaternic}, we have similarly  
\begin{align*}
{\E}_x\left[r_p^4(X_{t\land\sigma_{M\setminus G_n}})\right]&\leq r_p^4(x)+4(D+3)\int_0^t\mathfrak{D}(s)\d s+4
D\int_0^t{\E}_x\left[r_p^4(X_{s\land\sigma_{M\setminus G_n}})\right]\d s,
\end{align*}
and applying the Gronwall's inequality to the continuous function 
\begin{align*}
t\mapsto g(t):={\E}_x\left[r_p^4(X_{t\land\sigma_{M\setminus G_n}})\right]<+\infty
\end{align*}
under \eqref{eq:insidebehavior},  
we obtain \eqref{eq:quaternic}. 
\end{proof}

\begin{cor}\label{cor:conservative}
Suppose {\bf (B3)}. Then ${\bf X}^V$ is conservative. 
\end{cor}
\begin{proof}[\bf Proof] 
By \eqref{eq:nest2} and \eqref{eq:quadratic}, 
\begin{align*}
\E_x[r_p^2(X_{\zeta-}): t\geq\zeta, \zeta<+\infty]\,&\hspace{-0.1cm}\stackrel{\text{\tiny\eqref{eq:nest2}}}{=}
\E_x[\lim_{n\to\infty} r_p^2(X_{\sigma_{M\setminus G_n}}): t\geq\zeta, \zeta<+\infty]\\
&=\E_x[\lim_{n\to\infty} r_p^2(X_{t\land \sigma_{M\setminus G_n}}): t\geq\zeta, \zeta<+\infty]\\
&\leq\varliminf_{n\to\infty}\E_x[r_p^2(X_{t\land \sigma_{M\setminus G_n}})]\,\\
&\hspace{-0.2cm}\stackrel{\text{\text{\tiny\eqref{eq:quadratic}}}}{\leq}\mathfrak{D}(t)<+\infty.
\end{align*}
This and \eqref{eq:insidekilling} together imply  $\P_x(t\geq \zeta,\zeta<+\infty)=0$ for all $x\in M$. 
Therefore, we obtain $\P_x(\zeta<+\infty)=0$ for all $x\in M$.  
\end{proof}

Thanks to the conservativeness of ${\bf X}^V$ under {\bf (B3)}, hereafter, we may write $\tau_G=\sigma_{M\setminus G}$ for any relatively compact open subset $G$ of $M$. 
}}

\section{Bochner Identity}\label{sec:Bochner}
In this section, based on the
Bochner identity for smooth map $u:M\to N$ in terms of usual Laplace-Bertrami operator $\Delta$, we establish the
Bochner identity for $u:M\to N$ in terms of $\Delta_V$.
We denote the Levi-Civita connection on $M$ by $\nabla^M$.
Let $\{e_i\}_{i=1}^n$ be the canonical ONB of $TM$, i.e., $e_i=\left(\partial/\partial x^i \right)_x$, and $u^{-1}TN$ the induced
vector bundle from $TN$ by $u$, i.e., the fibre bundle over $M$ with fibre $T_{u(x)}N$ for $x\in M$. Moreover $\nabla^{u^{-1}TN}$ denotes the induced connection from the Levi-Civita
connection $\nabla^N$ on $N$. Let $\nabla$ be the connection over $TM^*\otimes u^{-1}TN$ induced from
the dual connection
$\nabla^*$ over $TM^*$ and the connection $\nabla^{u^{-1}TN}$ over
$u^{-1}TN$ (see \cite[(3.23)]{Nishikawa} for the definition of $\nabla$).
The \emph{Hessian} of $u:M\to N$ is the symmetric bilinear map
${\rm Hess}\,u:TM\times TM\to TN$ defined by
${\rm Hess}\,u(X,Y):=(\nabla_X\d u)(Y)=
\nabla_X^{u^{-1}TN}(\d u(Y))-\d u(\nabla_X^MY)$ for $X,Y\in\Gamma(TM)$
(see \cite[pp.~93--94]{Nishikawa} for definitions $\nabla':=\nabla^N$ and $'\nabla:=\nabla^{u^{-1}TN}$).
The tension field of $u$ is given by
$\tau(u):=\sum_{i=1}^n{\rm Hess}\,u(e_i,e_i)=\sum_{i=1}^n(\nabla_{e_i}\d u)(e_i)$ and the energy density of $u$ is given by $e(u):=\frac12|\d u|^2=\frac12\sum_{i=1}^n\langle\d u(e_i),\d u(e_i)\rangle$.
\begin{lem}[Bochner Identity]\label{lem:BochnerIdentity}
{\rm For any smooth map $u:M\to N$, we have
\begin{align}
\frac12\Delta_V|\d u|^2
&=|{\rm Hess}\,u|^2+\sum_{i=1}^n\left\langle\d u(e_i),\nabla_{e_i}^{u^{-1}TN}(\tau_V(u))\right\rangle
\notag\\
&\hspace{1cm}+\sum_{i,j=1}^n
{\Ric}_V^{\infty}(e_i,e_j)\left\langle\d u(e_i),\d u(e_j)\right\rangle\label{eq:BochnerIdentity}
\\
&\hspace{2cm}-
\sum_{i,j=1}^n
\left\langle{}^N{\rm Riem}(\d u(e_i),\d u(e_j))\d u(e_j),\d u(e_i)\right\rangle.\notag
\end{align}
Here ${}^N{\rm Riem}$ is the Riemannian curvature tensor for $(N,h)$.
}
\end{lem}
\begin{proof}[\bf Proof]
The assertion seems to be known for experts. 
For completeness we give a proof. 
The proof of \eqref{lem:BochnerIdentity} is based on
the following Bochner identity for smooth map $u$
\begin{align}
\frac12\Delta|\d u|^2
&=|{\rm Hess}\,u|^2+\sum_{i=1}^n\left\langle\d u(e_i),\nabla_{e_i}^{u^{-1}TN}(\tau(u))\right\rangle
\notag\\
&\hspace{1cm}+\sum_{i,j=1}^n
{\Ric}_g(e_i,e_j)\left\langle\d u(e_i),\d u(e_j)\right\rangle\label{eq:BochnerWeitzen}
\\&\hspace{2cm}-
\sum_{i,j=1}^n
\left\langle{}^N{\rm Riem}(\d u (e_i),\d u(e_j))\d u(e_j),\d u(e_i)\right\rangle.\notag
\end{align}
The Bochner identity \eqref{eq:BochnerWeitzen} for smooth map $u:M\to N$ is well-known for experts (see~\cite[Proposition~1.5]{AliMasRigo}, \cite[Corollary~2.10]{Loustau}). From \eqref{eq:BochnerWeitzen}, we can deduce \eqref{eq:BochnerIdentity} as follows:

\begin{align*}
\frac12\Delta_V|\d u|^2&=\frac12\Delta|\d u|^2-\frac12\langle V,\nabla |\d u|^2\rangle
\\
&\hspace{-0.1cm}\stackrel{\text{\tiny \eqref{eq:BochnerWeitzen}}}{=}|{\rm Hess}\, u|^2+\sum_{i=1}^n
\left\langle \d u(e_i),\nabla_{e_i}^{u^{-1}TN}(\tau_V(u))\right\rangle\\
&\hspace{1cm}+\sum_{i,j=1}^n
{\Ric}_V^{\infty}(e_i,e_j)\left\langle \d u(e_i),\d u(e_j)\right\rangle 
\\
&\hspace{2cm}-
\sum_{i,j=1}^n
\left\langle {}^N{\rm Riem}(\d u(e_i),\d u(e_j))\d u(e_j),\d u(e_i)\right\rangle\\
&\hspace{1cm}+\sum_{i=1}^n\left\langle \d u (e_i),\nabla_{e_i}^{u^{-1}TN}(\d u(V))\right\rangle\\
&\hspace{2cm}+\frac12\sum_{i,j=1}^n\left(\left\langle \nabla_{e_i}^MV,e_j\right\rangle+
\left\langle \nabla_{e_j}^MV,e_i\right\rangle \right)\left\langle \d u(e_i),\d u(e_j)\right\rangle\\
&\hspace{3cm}-\frac12\langle V,\nabla|\d u|^2\rangle.
\end{align*}
So it suffices to prove that the last three terms vanish. Observe
\begin{align*}
|\d u|^2=\sum_{i=1}^n|\d u(e_i)|^2.
\end{align*}
By definition of the connection $\nabla$ in $TM^*\otimes u^{-1}TN$,
\begin{align*}
\frac12V|\d u|^2&=\sum_{i=1}^n\langle \nabla_V\d u(e_i),\d u(e_i)\rangle
\\
&=\sum_{i=1}^n\left\langle \nabla \d u(V,e_i),\d u(e_i)\right\rangle\\
&
=\sum_{i=1}^n\left\langle \nabla_V^{u^{-1}TN}\d u(e_i)-\d u(\nabla_V^Me_i),\d u(e_i)\right\rangle\qquad \text{(by  \cite[(3.25)]{Nishikawa})}\\
&
=\sum_{i=1}^n\left\langle\d u(e_i),\nabla_V^{u^{-1}TN}\d u(e_i)\right\rangle-\sum_{i=1}^n
\left\langle\d u(\nabla_V^Me_i),\d u(e_i) \right\rangle.
\end{align*}
Then the sum of last three terms becomes
\begin{align*}
&\sum_{i=1}^n\left\langle\d u(e_i),\nabla_{e_i}^{u^{-1}TN}(\d u(V))\right\rangle-\sum_{i=1}^n\left\langle
\d u(e_i),\nabla_V^{u^{-1}TN}\d u(e_i)\right\rangle\\
&\hspace{1cm}+\sum_{i=1}^n\left\langle \d u(\nabla_V^Me_i),\d u(e_i)\right\rangle-\sum_{i,j=1}^n\left\langle \nabla_{e_i}^MV,e_j\right\rangle\langle \d u(e_i),\d u(e_j)\rangle
\\
&=\sum_{i=1}^n\left\langle \d u(e_i),\nabla_{e_i}^{u^{-1}TN}(\d u(V))-\nabla_V^{u^{-1}TN}\d u(e_i)\right\rangle\\
&\hspace{1cm}+\sum_{i=1}^n\left\langle\d u(\nabla_V^Me_i),\d u(e_i)\right\rangle-
\sum_{i,j=1}^n\left\langle\nabla_{e_i}^MV,e_j\right\rangle\left\langle\d u(e_i),\d u(e_j)\right\rangle
\\
&=\sum_{i=1}^n\left\langle\d u(e_i),\d u([e_i,V])
\right\rangle\qquad(\text{by \cite[Lemma~3.7]{Nishikawa}})\\
&\hspace{1cm}+\sum_{i=1}^n\left\langle\d u(\nabla_V^Me_i),\d u(e_i)\right\rangle-
\sum_{i,j=1}^n\left\langle\nabla_{e_i}^MV,e_j\right\rangle\left\langle\d u(e_i),\d u(e_j)\right\rangle\\
&=\sum_{i=1}^n\left\langle\d u(e_i),\d u(\nabla_{e_i}^MV-\nabla_V^Me_i)
\right\rangle\\
&\hspace{1cm}+\sum_{i=1}^n\left\langle\d u(\nabla_V^Me_i),\d u(e_i)\right\rangle-\sum_{i=1}^n\left\langle\d u(e_i),\d u(\nabla_{e_i}^MV)\right\rangle\\&=0.
\end{align*}
\end{proof}
We need the following lemma to prove Corollary~\ref{cor:BochnerIdentity}.
\begin{lem}\label{lem:CauchySchwarz}
Let $H$ be a real separable Hilbert space and suppose that $h_{ij}\in H$ satisfies $h_{ij}=h_{ji}$ for $i,j=1,2,\cdots, n$. Then we have
\begin{align*}
\sum_{i,j=1}^n\|h_{ij}\|_H^2\geq\frac{1}{n}\left\|\sum_{i=1}^nh_{ii}\right\|_H^2.
\end{align*}
If further the symmetric matrix $(h_{ij})\in H^n\otimes H^n$ has a null eigenvalue
$0_H\in H$ with multiplicity $k\in\{1,2,\cdots, n-1\}$, then  we have
\begin{align*}
\sum_{i,j=1}^n\|h_{ij}\|_H^2\geq\frac{1}{n-k}\left\|\sum_{i=1}^nh_{ii}\right\|_H^2.
\end{align*}
\end{lem}
\begin{proof}[\bf Proof]
The first assertion can be easily confirmed for the case $H=\R$. In fact, in this case any real symmetric matrix $A=(a_{ij})$ satisfies $\|A\|^2={\rm Tr}(A^2)\geq\frac{{\rm Tr}(A)^2}{n}$. For general case, letting $\{e_{\ell}\}_{\ell=1}^{\infty}$ be the complete orthonormal base in $H$, the assertion follows from the Parseval's identity $\|h\|_H^2=\sum_{\ell=1}^{\infty}\langle h,e_{\ell}\rangle_H^2$ and the assertion for $H=\R$. Next we prove the second assertion.
In view of the Fourier expansion $h=\sum_{\ell=1}^{\infty}\langle h,e_{\ell}\rangle_He_{\ell}$ in $H$, we can prove
\begin{align*}
{\rm Ker}(h_{ij})=\bigcap_{\ell=1}^{\infty}{\rm Ker}(\langle h_{ij},e_{\ell}\rangle_H),
\end{align*}
where
\begin{align*}
{\rm Ker}(h_{ij}):=\{\phi\in \R^n\mid (h_{ij})\phi=0_{H^n}\}, \quad
{\rm Ker}(\langle h_{ij},e_{\ell}\rangle_H):=\{\phi\in \R^n\mid (\langle h_{ij},e_{\ell}\rangle_H)\phi=0_{\R^n}\}.
\end{align*}
Therefore, if $(h_{ij})$ has a null eigenvalue $0_H$ with multiplicity $k$, then
the real symmetric matrix $(\langle h_{ij},e_{\ell}\rangle_H)$  has the null eigenvalue $0$
with multiplicity $k$ for all $\ell\in\N$. The converse also holds.
The rest of the proof is similar to the proof of the first assertion.
\end{proof}
\begin{cor}\label{cor:BochnerIdentity}
Suppose that ${\Ric}_V^m\geq 0$ with $m\in\,[-\infty,\,0\,]\,\cup\,[\,n,+\infty\,]$ and ${\rm Sect}_N\leq0$. Then
any $V$-harmonic map $u:M\to N$ satisfies
\begin{align}
\Delta_V|\d u|^2\geq\frac{2m}{n(m-n)}|\d u(V)|^2\geq0.\label{eq:BochnerIneq}
\end{align}
If further the symmetric matrix $({\rm Hess}\,u(e_i,e_j))_{ij}$ for $V$-harmonic map $u:M\to N$ admits a null eigenvalue with multiplicity $k\in\{1,2,\cdots, n-1\}$ independent of the choice of fibre, then
$\Delta_V|\d u|^2\geq\frac{2(m-k)}{(n-k)(m-n)}|\d u(V)|^2\geq0$ under
${\Ric}_V^m\geq 0$ with $m\in ]-\infty,\,k\,]$ and ${\rm Sect}_N\leq0$.
\end{cor}
\begin{proof}[\bf Proof]
We may assume $m\in]-\infty,\,0\,]$ by \eqref{eq:order}.
By Lemma~\ref{lem:CauchySchwarz},
\begin{align*}
|{\rm Hess}\,u|^2=\sum_{i,j=1}^n|{\rm Hess}\,u(e_i,e_j)|^2
\geq\frac{1}{n}\left|\sum_{i=1}^n{\rm Hess}\,u(e_i,e_i) \right|^2
=\frac{|\tau(u)|^2}{n}.
\end{align*}
Since $u:M\to N$ is $V$-harmonic, $\Ric_V^m\geq0$ and ${\rm Sect}_N\leq0$, we have
\begin{align*}
\frac12\Delta_V|\d u|^2&\geq |{\rm Hess}\,u|^2+\sum_{i,j=1}^n
{\Ric}_V^{\infty}(e_i,e_j)\langle \d u(e_i),\d u(e_j)\rangle\\
&=|{\rm Hess}\,u|^2+\sum_{i,j=1}^n
{\Ric}_V^m(e_i,e_j)\langle \d u(e_i),\d u(e_j)\rangle \\
&\hspace{1cm}+\frac{1}{m-n}\sum_{i,j=1}^n \langle V,e_i\rangle\langle V,e_j\rangle
\langle \d u(e_i),\d u(e_j)\rangle
\end{align*}
\begin{align*}
&\geq |{\rm Hess}\,u|^2+\frac{1}{m-n}\sum_{i,j=1}^n \langle V,e_i\rangle\langle V,e_j\rangle
\langle\d u(e_i),\d u(e_j)\rangle\\
&\geq \frac{|\tau(u)|^2}{n}+\frac{1}{m-n}|\d u(V)|^2
\\
&= \frac{|\d u(V)|^2}{n}+\frac{1}{m-n}|\d u(V)|^2\qquad (\text{by }\tau_V(u)=0)
\\&
=\frac{m}{n(m-n)}|\d u (V)|^2\geq0.
\end{align*}
In the second inequality above, we use the fact that any $n\times n$ symmetric non-negative definite real matrix $A=(a_{ij})$ satisfies $\sum_{i,j=1}^na_{ij}\langle h_i,h_j\rangle_H\geq0$ for any real Hilbert space $(H,\langle \cdot,\cdot\rangle)$ 
and $h_i\in H$ ($1\leq i\leq n$). If further $({\rm Hess}\,u(e_i,e_j))_{ij}$
for $V$-harmonic map $u:M\to N$ admits a null eigenvalue with multiplicity $k\in\{1,2,\cdots, n-1\}$ independent of the choice of fibre, then
$|{\rm Hess}\,u|^2\geq\frac{|\tau(u)|^2}{n-k}$, consequently
$\Delta_V|\d u|^2\geq \frac{2(m-k)}{(n-k)(m-n)}|\d u(V)|^2\geq0$ under $m\in]-\infty,\,k\,]$.
\end{proof}
\begin{remark}\label{rem:BochnerIdentity}
{\rm
\begin{enumerate}
\item For any smooth function $u:M\to\R$, the assertion of Lemma~\ref{lem:BochnerIdentity} is well-known:
\begin{align}
\frac12\Delta_V|\nabla u|^2=|{\rm Hess}\,u|^2+\langle \nabla\Delta_Vu,\nabla u\rangle+
{\Ric}_V^{\infty}(\nabla u,\nabla u).
\label{eq:BochnerIdentity*}
\end{align}
From \eqref{eq:BochnerIdentity*},
for $m\in\,[-\infty,\,0\,]\,\cup\,[\,n,+\infty\,]$, we have
\begin{align}
\frac12\Delta_V|\nabla u|^2\geq \frac{(\Delta_Vu)^2}{n}+\frac{2\Delta_Vu\langle V,\nabla u\rangle}{n}+\langle\nabla\Delta_Vu,\nabla u\rangle+{\Ric}_V^m(\nabla u,\nabla u),\label{eq:BochnerIdentity**}
\end{align}
which has been known for $m\in \,[\,n,+\infty\,]$. The inequality  \eqref{eq:BochnerIdentity**} also yields that any $V$-harmonic function $u:M\to\R$ satisfies $\Delta_V|\nabla u|^2\geq0$ under
${\Ric}_V^m\geq0$ with $m\in\,[-\infty,\,0\,]\,\cup\,[\,n,+\infty\,]$.
\item When $m\in \,[\,n, +\infty\,]$,  the third named author proved the following Bochner inequality  for $V=\nabla f$ (see p.~1303 in  \cite{Xdli:Liouville})
\begin{align}
\frac{1}{2} \Delta_V |\nabla u|^2 \geq
\Ric^{m}_V(\nabla u,\nabla u)+\frac{(\Delta_V u)^2}{m}+
\langle\nabla\Delta_V u,\nabla u\rangle. \label{eq:BochnernegativeN}
\end{align}
In ~\cite{Oh<0}, Ohta extended  $(\ref{eq:BochnernegativeN})$ to the case $m\in\,[\,-\infty,\,0\,[$, which is a special case of \eqref{eq:BochnerStrong}  below for $\eps=1$.
On the other hand, Kolesnikov-Milman~\cite{KolMil:2017} also
obtained \eqref{eq:BochnernegativeN} including the case $m=0$ with the convention $1/0-:=-\infty$.
Our inequality \eqref{eq:BochnerIdentity**} is different from  \eqref{eq:BochnernegativeN} at least for the case $m=0$.
The inequality \eqref{eq:BochnernegativeN} can be proved
for general $V$. Indeed, by the same argument as used for the proof of 
$(\ref{eq:BochnernegativeN})$  in  \cite[p.~1303]{Xdli:Liouville}  for $V=\nabla f$ and $m\in [\,n,\, +\infty\,]$,
the Bochner identity
\eqref{eq:BochnerWeitzen} for smooth function $u:M\to\R$ and the elementary inequality
$(a+b)^2\geq\frac{a^2}{1+\alpha}-\frac{b^2}{\alpha}$ for $a,b\in\R$ with $\alpha\in\,]-\infty,\,-1\,[\,\cup\,]\,0,\,+\infty\,[$
together show
\begin{align*}
\frac12\Delta_V|\nabla u|^2&\geq {\Ric}_V^{\infty}(\nabla u,\nabla u)+\frac{(\Delta_V u+\langle V,\nabla u\rangle)^2}{n}+\langle\nabla\Delta_V u,\nabla u\rangle\\
&\geq {\Ric}_V^m(\nabla u,\nabla u)+\frac{\langle V,\nabla u\rangle^2}{m-n}
+\frac{(\Delta_V u)^2}{n(1+\alpha)}-\frac{\langle V,\nabla u\rangle^2}{n\alpha}+\langle \nabla\Delta_V u,\nabla u\rangle\\
&={\Ric}_V^m(\nabla u,\nabla u)+\frac{(\Delta_V u)^2}{m}+\langle\nabla\Delta_V u,\nabla u\rangle,
\end{align*}
where we take $\alpha=\frac{m}{n}-1<-1$ for $m\in]-\infty,\,0\,[$ in the last equality.
\item When $V=\nabla f$ and $u:M\to \R$ is a smooth function,
we can deduce the following inequality for $m\in\,[\,-\infty,\,0\,]\,\cup\,[\,n,+\infty\,]$ and $\eps\in\R$ with
\begin{align*}
\varepsilon=0\, \text{ for }\, m=0,\,\,\, \varepsilon\in \left]-\sqrt{\frac{m}{m-n}},\sqrt{\frac{m}{m-n}} \right[  \, \text{ for }\,  m\ne 0,n,\,\,\, \varepsilon\in\R \, \text{ for } \, m=n:
\end{align*}
\begin{align}\notag
\frac{1}{2} \Delta_f |\nabla u|^2 & =   \Ric^{m}_{f}(\nabla u,\nabla u)+ \frac{1}{n}\left(1-\eps^2\frac{m-n}{m} \right)\,(\Delta_f u)^2\\ \notag
 &\hspace{3cm}+e^{-\frac{2(1-\eps)f}{n}}\langle\nabla (e^{\frac{2(1-\eps)f}{n}} \, \Delta_f u), \nabla u \rangle  \\ \notag
   &\hspace{1cm} +\frac{1}{n}\left(  \sqrt{\frac{m}{m-n}}\langle\nabla f,\nabla u \rangle +\eps \sqrt{\frac{m-n}{m}}\Delta_f u  \right)^2+\left\|\nabla^2u-\frac{\Delta u}{n}g \right\|^2\\
&\geq \Ric^{m}_f(\nabla u,\nabla u)+\frac{1}{n}\left(1-\eps^2\frac{m-n}{m} \right)(\Delta_f u)^2\notag\\
&\hspace{3cm}+e^{-\frac{2(1-\eps)f}{n}}\langle \nabla (e^{\frac{2(1-\eps)f}{n}} \, \Delta_f u), \nabla u \rangle.\label{eq:BochnerStrong}
\end{align}
The inequality \eqref{eq:BochnerStrong}
is meaningful even if $m=0$ and $\eps=0$ with the convention $\eps^2\frac{m-n}{m}=0$.
Our inequality \eqref{eq:BochnerIdentity**} for $V=\nabla f$ is a consequence of \eqref{eq:BochnerStrong} for $\eps=0$.
\item When $u:M\to N$ is $V$-harmonic, \eqref{eq:BochnerIdentity} is also proved in \cite[Lemma~1]{ChenJostQiu}.
\end{enumerate}
}
\end{remark}
\begin{lem}\label{lem:LapalcianDistaHarm}
Suppose that $(N,h)$ is a Cartan-Hadamard manifold.
Let $u:M\to N$ be a $V$-harmonic map. Then $\Delta_V {\color{black}{{\sf d}}}_N^{\,2}(u,o)\geq 2|\d u|^2$ for any $o\in N$.
\end{lem}
\begin{proof}[\bf Proof]
By using \cite[1.177]{AliMasRigo},
the Hessian of the composite map of ${\color{black}{{\sf d}}}_N^{\,2}(\cdot,o):N\to\R$ and $u:M\to N$ is given by
\begin{align}
{\rm Hess}\,{\color{black}{{\sf d}}}_N^{\,2}(u,o)(e_i,e_j)={\rm Hess}^{N}{\color{black}{{\sf d}}}_N^{\,2}(\cdot,o)(\d u(e_i),\d u(e_j))+\d({\color{black}{{\sf d}}}_N^{\,2}(\cdot,o))({\rm Hess}\,u(e_i,e_j)).\label{eq:Hessian}
\end{align}
By taking a trace to \eqref{eq:Hessian}, we have
\begin{align*}
\Delta {\color{black}{{\sf d}}}_N^{\,2}(u,o)=\sum_{i=1}^n{\rm Hess}^{N} {\color{black}{{\sf d}}}_N^{\,2}(\cdot,o)(\d u(e_i),\d u(e_i))+\d {\color{black}{{\sf d}}}_N^{\,2}(\cdot,o)(\tau(u)).
\end{align*}
Thus we obtain
\begin{align*}
\Delta_V {\color{black}{{\sf d}}}_N^{\,2}(u,o)&=\sum_{i=1}^n{\rm Hess}^{N} {\color{black}{{\sf d}}}_N^{\,2}(\cdot,o)(\d u(e_i),\d u(e_i))+\d {\color{black}{{\sf d}}}_N^{\,2}(\cdot,o)(\tau_V(u))\\
&=\sum_{i=1}^n{\rm Hess}^{N} {\color{black}{{\sf d}}}_N^{\,2}(\cdot,o)(\d u(e_i),\d u(e_i))\\
&\geq 2\sum_{i=1}^n\langle \d u(e_i),\d u(e_i)\rangle=2|\d u|^2.
\end{align*}
In the last inequality above, we use the Greene-Wu Hessian comparison theorem under
${\rm Sect}_N\leq0$ (see \cite[Theorem~A]{GreeneWu}, \cite[p.~227]{SiuYau}, \cite[Chapter IV, Lemma~2.9]{Sakai}),
where we use the completeness and simple connectedness of $(N,h)$.
Note that ${\rm Hess}^{N}{\color{black}{{\sf d}}}_N(\cdot,o)^2(e_{\alpha}',e_{\beta}')=2\delta_{\alpha\beta}$ for $o\in N=\R^{\ell}$ and $\alpha,\beta=1,2,\cdots,\ell$ with $e_{\alpha}'=\left(\partial/\partial y^{\alpha} \right)_y$, $y\in\R^{\ell}$.
\end{proof}
\begin{remark}\label{rem:simplyconnected}
{\rm In the assertion of Lemma~\ref{lem:LapalcianDistaHarm}, we can relax the simple connectedness of $(N,h)$.
 Precisely to say, if ${\rm Sect}_N\leq0$ and ${\rm Im}(u)\cap {\rm Cut}(o)=\emptyset$, then
$\Delta_V{\color{black}{{\sf d}}}_N^2(u,o)\geq 2|\d u|^2$. This is a consequence of Greene-Wu's Hessian comparison theorem (e.g.~\cite[Chapter IV, Lemma~2.9]{Sakai}).
}
\end{remark}

\section{Polynomial growth of differential map for $V$-harmonic maps}\label{sec:quadraticgrowth}
In this section, we show that for each $i=1,2,3$, if a smooth $V$-harmonic map $u:M\to N$ has a growth condition 
{\bf (G{\boldmath$i$})}
under 
{\bf (B3)} and ${\rm Ric}_V^m\geq0$ with $m\in[-\infty,\,0\,]\,\cup\,[\,n,\,+\infty\,]$, then 
the differential map $\d u:TM\to TN$ also has the same growth {\bf (G{\boldmath$i$})}. 
For this, we need the following lemma. 

\begin{lem}\label{lem:lineargrowth}
Assume ${\rm Ric}_V^m\geq0$ with $m\in[-\infty,\,0\,]\,\cup\,[\,n,\,+\infty\,]$. 
Suppose that $u:M\to N$ is a smooth $V$-harmonic map  and {\bf (B3)}. 
Then for each $p\in M$, there exist $C=C(p,D)>0$ and $A=A(p,D)>0$ such that  
\begin{align*}
\sup_{x\in B_a(p)}|\d u|^2(x)\leq C(m_u(2a)+1)^2\qquad \text{ for }\qquad a>A. 
\end{align*}
\end{lem}
Here $m_u(2a):=\sup_{B_{2a}(p)}{\color{black}{{\sf d}}}_N(u,o)$. In particular, for each $i=1,2,3$ 
if {\bf (G{\boldmath$i$})} holds for $u$, then it holds for $|\d u|$.
\begin{proof}[\bf Proof]
Let $\tilde{\varphi}(r)$ be a $C^2$-function defined on $[\,0,\,+\infty\,[$ such that $\tilde{\varphi}(r)>0$ for $r\in[0,2[$ and 
\begin{align*}
\tilde{\varphi}(r):=\left\{\begin{array}{cl}1 & \quad\text{ if }\quad r\in[\,0,\,1\,], \\0 & \quad \text{ if }\quad r\in[\,2,+\infty\,[\end{array}\right.
\end{align*}
with 
\begin{align*}
0\geq\frac{\tilde{\varphi}'(r)}{\sqrt{\tilde{\varphi}(r)}}\geq -C_1
\end{align*}
and
\begin{align*}
\tilde{\varphi}''(r)\geq -C_2
\end{align*}
for some constants $C_1,C_2>0$ (see \cite[proof of Theorem~1.2]{LiYau} or 
\cite[proof of Theorem~1.5]{Xdli:Liouville}). We set 
\begin{align*}
\varphi(x):=\tilde{\varphi}\left(\frac{r_p(x)}{a} \right).
\end{align*}
We then see that on $B_{2a}(p)\setminus {\rm Cut}(p)$ 
\begin{align}
|\nabla \varphi|\leq\frac{C_1}{a}\quad \text{ and }\quad\frac{|\nabla \varphi|^2}{\varphi}\leq \frac{C_1^2}{a^2},\label{eq:gradestvarphi}
\end{align}
and
\begin{align}
\Delta_V\varphi=\frac{\tilde{\varphi}'(r_p/a)\Delta_Vr_p}{a}+\frac{\tilde{\varphi}''(r_p/a)|\nabla r_p|^2}{a^2}\geq-\frac{C_1D(1+2a^2)+C_2}{a^2}\label{eq:Laplacianvarphi}
\end{align}
on $B_{2a}(p)$. Here we use {\bf (B3)}. 
Set $b:=\sqrt{5}(m_u(2a)+1)$, $F(x)=\frac{|\d u|^2(x)}{b^2-{\color{black}{{\sf d}}}_N^{\,2}(u(x),o)}$ and $G:=\varphi F$. 
Note that $F\geq0$ on $B_{2a}(p)$ because $b^2-{\color{black}{{\sf d}}}_N^{\,2}(u(x),o)\geq 5>0$ on $B_{2a}(p)$. 
Moreover, 
\begin{align}
\frac{{\color{black}{{\sf d}}}_N^{\,2}(u,o)}{b^2-{\color{black}{{\sf d}}}_N^{\,2}(u,o)}\leq \frac{m_u(2a)^2}{b^2-m_u(2a)^2}\leq \frac{1}{4}\quad\text{ on }\quad B_{2a}(p).\label{eq:FracEst}
\end{align} 

Then, $G$ has a maximum at $\bar{x}\in B_{2a}(p)$. 
At this point $r_p$ may not be twice 
differentiable. We will remedy this by the Calabi's argument (see \cite{Calabi:strongmax}): 
Let $\gamma:[0,r_p(\bar{x})]\to M$ be the minimal geodesic joining  $\gamma_0=p$ and 
$\gamma_{r_p(\bar{x})}=\bar{x}$. Let $\eps>0$ be small. Along $\gamma$, $\gamma_{\eps}$ is not conjugate to $\bar{x}$. So there exists a geodesic cone $\mathscr{C}$ with vertex $\gamma_{\eps}$ and 
contains a neighborhood of $\bar{x}$ such that $x\mapsto r_{\gamma_{\eps}}(x)$ is smooth inside 
$\mathscr{C}$. Let $\tilde{r}_{\eps}(x):=r_{\gamma_{\eps}}(x)+\eps$. Then triangle 
inequality states that 
\begin{align}
r_p(x)\leq\tilde{r}_{\eps}(x)\quad\text{ for }\quad x\in\mathscr{C}, \qquad r_p(\bar{x})=\tilde{r}_{\eps}(\bar{x}).
\label{eq:Calabi}
\end{align}
Then $\tilde{G}_{\eps}:=\tilde{\varphi}\left(\frac{\tilde{r}_{\eps}}{a} \right)F$ is smooth near $\bar{x}$ and also attains the same maximum value at $\bar{x}$. Letting $\eps\to0$, we may assume without loss of generality that 
$r_p$ is already smooth near $\bar{x}$.

If $F(\bar{x})=0$, we have $|\d u|\equiv0$ on $B_{2a}(p)$. Then we have the conclusion. 
So we may assume $F(\bar{x})>0$, hence $|\d u|(\bar{x})>0$. 
At $\bar{x}$, we have 
\begin{align*}
0\geq \Delta_V(\varphi F)\quad\text{ and }\quad \nabla(\varphi F)=0.
\end{align*}
The latter one is equivalent to $0=\frac{\nabla\varphi}{\varphi}+\frac{\nabla F}{F}$ at $\bar{x}$, i.e., 
\begin{align}
\frac{\nabla {\color{black}{{\sf d}}}_N^{\,2}(u,o)}{b^2-{\color{black}{{\sf d}}}_N^{\,2}(u,o)}=-\frac{\nabla\varphi}{\varphi}-\frac{\nabla |\d u|^2}{|\d u|^2}\label{eq:Plugg}
\end{align}
at $\bar{x}$. 
These imply that at $\bar{x}$ with \eqref{eq:Laplacianvarphi}
\begin{align*}
0&\geq (\Delta_V\varphi)F+2\langle\nabla\varphi,\nabla F\rangle+\varphi(\Delta_VF)\\
&\geq -\frac{C_1D(1+2a^2)+C_2}{a^2}F+2\langle\nabla\varphi,\nabla(\varphi F)\rangle-2\frac{|\nabla \varphi|^2}{\varphi}F\\
&\hspace{1cm}+\varphi\left(
\frac{\Delta_V{\color{black}{{\sf d}}}_N^{\,2}(u,o)}{b^2-{\color{black}{{\sf d}}}_N^{\,2}(u,o)}
+2\left\langle \nabla|\d u|^2, \nabla\frac{1}{b^2-{\color{black}{{\sf d}}}_N^{\,2}(u,o)}\right\rangle
+|\d u|^2\Delta_V\frac{1}{b^2-{\color{black}{{\sf d}}}_N^{\,2}(u,o)} 
\right)\\
&= -\frac{C_1D(1+2a^2)+C_2}{a^2}F-2\frac{|\nabla \varphi|^2}{\varphi}F\\
&\hspace{2cm}+\varphi \left\{
\frac{\Delta_V{\color{black}{{\sf d}}}_N^{\,2}(u,o)}{b^2-{\color{black}{{\sf d}}}_N^{\,2}(u,o)}
+2\left\langle\nabla|\d u|^2, \frac{\nabla {\color{black}{{\sf d}}}_N^{\,2}(u,o)}{(b^2-{\color{black}{{\sf d}}}_N^{\,2}(u,o))^2}\right\rangle \right.
\\
&\hspace{4cm}\left.
+|\d u|^2\left(\frac{2|\nabla {\color{black}{{\sf d}}}_N^{\,2}(u,o)|^2}{(b^2-{\color{black}{{\sf d}}}_N^{\,2}(u,o))^3}+\frac{\Delta_V{\color{black}{{\sf d}}}_N^{\,2}(u,o)}{(b^2-{\color{black}{{\sf d}}}_N^{\,2}(u,o))^2} \right)
\right\}.
\end{align*}
Substituting \eqref{eq:Plugg} into the the second term in the brace above, 
the last equals 
\begin{align*}
\hspace{-3cm}-&\frac{C_1D(1+2a^2)+C_2}{a^2}F-2\frac{|\nabla \varphi|^2}{\varphi}F\\
&\hspace{1cm}+\varphi \frac{\Delta_V{\color{black}{{\sf d}}}_N^{\,2}(u,o)}{b^2-{\color{black}{{\sf d}}}_N^{\,2}(u,o)}
-\frac{2\varphi}{b^2-{\color{black}{{\sf d}}}_N^{\,2}(u,o)}\left(\frac{|\nabla |\d u|^2|^2}{|\d u|^2}
-\frac{\langle\nabla|\d u|^2,\nabla\varphi\rangle}{\varphi}\right)\\
&\hspace{2cm}+\varphi\frac{|\d u|^2}{b^2-{\color{black}{{\sf d}}}_N^{\,2}(u,o)}\left(\frac{2|\nabla {\color{black}{{\sf d}}}_N^{\,2}(u,o)|^2}{(b^2-{\color{black}{{\sf d}}}_N^{\,2}(u,o))^2}+\frac{\Delta_V {\color{black}{{\sf d}}}_N^{\,2}(u,o)}{b^2-{\color{black}{{\sf d}}}_N^{\,2}(u,o)} \right).
\end{align*}
By Lemma~\ref{lem:LapalcianDistaHarm}, we know $\Delta_V{\color{black}{{\sf d}}}_N^2(u,o)\geq 2|\d u|^2\geq0$, the last 
equation is greater than 
\begin{align*}
&\hspace{-3cm}-\frac{C_1D(1+2a^2)+C_2}{a^2}F-2\frac{|\nabla \varphi|^2}{\varphi}F\\
&\hspace{-1.6cm}-2\varphi F\frac{|\nabla|\d u|^2|^2}{|\d u|^4}-2\varphi F\left\langle\frac{\nabla|\d u|^2}{|\d u|^2},\frac{\nabla\varphi}{\varphi} \right\rangle\\
&\hspace{-0.6cm}+\varphi F\left( \frac{8{\color{black}{{\sf d}}}_N^{\,2}(u,o)|\d u|^2}{(b^2-{\color{black}{{\sf d}}}_N^{\,2}(u,o))^2}+\frac{2|\d u|^2}{b^2-{\color{black}{{\sf d}}}_N^{\,2}(u,o)}\right).
\end{align*}
Multiplying $\varphi(\bar{x})$ in both sides, we have 
\begin{align*}
0&\geq -\frac{C_1D(1+2a^2)+C_2}{a^2}\varphi F-2\frac{|\nabla \varphi|^2}{\varphi}\varphi F\\
&\hspace{1cm}-2\varphi\cdot\varphi F\cdot\frac{|\nabla|\d u|^2|^2}{|\d u|^4}-2\varphi F
\left\langle\frac{\nabla|\d u|^2}{|\d u|^2},\nabla \varphi \right\rangle\\
&\hspace{2cm}+\varphi\cdot\varphi F\left(\frac{8{\color{black}{{\sf d}}}_N^{\,2}(u,o)}{b^2-{\color{black}{{\sf d}}}_N^{\,2}(u,o)}F+2F \right)\\
&\geq -\frac{C_1D(1+2a^2)+C_2}{a^2}\varphi F-2\frac{|\nabla \varphi|^2}{\varphi}\varphi F\\
&\hspace{1cm}-2\varphi\cdot\varphi F\cdot\frac{|\nabla|\d u|^2|^2}{|\d u|^4}-2\varphi F
\frac{|\nabla|\d u|^2|}{|\d u|^2}|\nabla \varphi| \\
&\hspace{2cm}+\varphi\cdot\varphi F\left(\frac{8{\color{black}{{\sf d}}}_N^{\,2}(u,o)}{b^2-{\color{black}{{\sf d}}}_N^{\,2}(u,o)}F+2F \right).
\end{align*}
Applying \eqref{eq:gradestvarphi} and \eqref{eq:Plugg}, we have 
\begin{align*}
0&\geq -\frac{C_1D(1+2a^2)+C_2}{a^2}\varphi F-2\frac{C_1^2}{a^2}\varphi F\\
&\hspace{1cm}-2\varphi^2 F\left(\frac{4{\color{black}{{\sf d}}}_N^{\,2}(u,o)|\d u|^2}{b^2-{\color{black}{{\sf d}}}_N^{\,2}(u,o))^2}+\frac{4{\color{black}{{\sf d}}}_N(u,o)|\d u||\nabla\varphi|}{(b^2-{\color{black}{{\sf d}}}_N^{\,2}(u,o))\varphi}+\frac{|\nabla\varphi|^2}{\varphi^2} \right)\\
&\hspace{2cm}-2\varphi F\left(\frac{|\nabla\varphi|^2}{\varphi}+\frac{2{\color{black}{{\sf d}}}_N(u,o)|\d u|}{b^2-{\color{black}{{\sf d}}}_N^{\,2}(u,o)} \right)\\
&\hspace{3cm}+\varphi F\left(\frac{8{\color{black}{{\sf d}}}_N^{\,2}(u,o)}{b^2-{\color{black}{{\sf d}}}_N^{\,2}(u,o)}\varphi F+2\varphi F \right)\\
&\geq -\frac{C_1D(1+2a^2)+C_2+2C_1^2}{a^2}\varphi F\\
&\hspace{1cm}-2\varphi^2 F\left(4\frac{{\color{black}{{\sf d}}}_N^{\,2}(u,o)}{b^2-{\color{black}{{\sf d}}}_N^{\,2}(u,o)}F+4\sqrt{\frac{{\color{black}{{\sf d}}}_N^{\,2}(u,o)}{b^2-{\color{black}{{\sf d}}}_N^{\,2}(u,o)}}\sqrt{F}\frac{|\nabla\varphi|}{\varphi}+\frac{|\nabla\varphi|^2}{\varphi^2} \right)\\
&\hspace{2cm}-2\varphi F\cdot\frac{|\nabla\varphi|^2}{\varphi}
-4\varphi F|\nabla\varphi|\frac{{\color{black}{{\sf d}}}_N(u,o)|\d u|}{b^2-{\color{black}{{\sf d}}}_N^{\,2}(u,o)}\\
&\hspace{3cm}+\varphi F\left(\frac{8{\color{black}{{\sf d}}}_N^2{\,}(u,o)}{b^2-{\color{black}{{\sf d}}}_N^{\,2}(u,o)}\varphi F+2\varphi F \right).
\end{align*}
Then 
\begin{align*}
0&\geq -\frac{C_1D(1+2a^2)+C_2+2C_1^2}{a^2}\varphi F\\
&\hspace{1cm}-8(\varphi F)^2\frac{{\color{black}{{\sf d}}}_N^{\,2}(u,o)}{b^2-{\color{black}{{\sf d}}}_N^{\,2}(u,o)}-8(\varphi F)\sqrt{\frac{{\color{black}{{\sf d}}}_N^{\,2}(u,o)}{b^2-{\color{black}{{\sf d}}}_N^{\,2}(u,o)}}\sqrt{\varphi F}\frac{|\nabla\varphi|}{\sqrt{\varphi}}-4(\varphi F)\frac{|\nabla\varphi|^2}{\varphi}\\
&\hspace{2cm}-4(\varphi F)\cdot\frac{|\nabla\varphi|}{\sqrt{\varphi}}\sqrt{\frac{{\color{black}{{\sf d}}}_N^{\,2}(u,o)}{b^2-{\color{black}{{\sf d}}}_N^{\,2}(u,o)}}\sqrt{\varphi F}+\varphi F\left(8\frac{{\color{black}{{\sf d}}}_N^{\,2}(u,o)}{b^2-{\color{black}{{\sf d}}}_N^{\,2}(u,o)}\varphi F+2\varphi F \right)\\
&=-\frac{C_1D(1+2a^2)+C_2+2C_1^2}{a^2}\varphi F\\
&\hspace{1cm}-4(\varphi F)\frac{|\nabla\varphi|^2}{\varphi}-12(\varphi F)\cdot\frac{|\nabla\varphi|}{\sqrt{\varphi}}\sqrt{\frac{{\color{black}{{\sf d}}}_N^{\,2}(u,o)}{b^2-{\color{black}{{\sf d}}}_N^{\,2}(u,o)}}\sqrt{\varphi F}+2(\varphi F)^2.
\end{align*}
Applying \eqref{eq:gradestvarphi} and \eqref{eq:FracEst}, and   
dividing by $G(\bar{x})=\varphi(\bar{x})F(\bar{x})>0$, we have at $\bar{x}$
\begin{align*}
0&\geq -\frac{C_1D(1+2a^2)+C_2+6C_1^2}{a^2}-\frac{6C_1}{a}\sqrt{\varphi F}+
2\varphi F\\
&\geq -\frac{C_1D(1+2a^2)+C_2+6C_1^2}{a^2}-\frac{6C_1}{a}\sqrt{\varphi F}+
\varphi F.
\end{align*}
This implies that 
\begin{align*}
\sqrt{G(\bar{x})}\leq 
\sqrt{\frac{12C_1^2}{a^2}+\frac{C_1D(1+2a^2)+C_2+6C_1^2}{a^2}}
=O(1).
\end{align*}
Therefore, 
\begin{align*}
\frac{1}{b^2}\sup_{B_a(p)}|\d u|^2\leq\sup_{B_a(p)}\frac{\varphi|\d u|^2}{b^2-{\color{black}{{\sf d}}}_N^{\,2}(u,o)}\leq 
\sup_{B_{2a}(p)}\frac{\varphi|\d u|^2}{b^2-{\color{black}{{\sf d}}}_N^{\,2}(u,o)}=G(\bar{x})=O(1),
\end{align*}
hence 
\begin{align*}
\sup_{B_a(p)}|\d u|^2\leq b^2O(1)=5(m_u(2a)+1)^2O(1).
\end{align*}
\end{proof}

Finally, we show the following lemma for the upper bound. 

\begin{lem}\label{lem:squaredistance*}
We have the following:
\begin{enumerate}
\item Suppose {\bf (B1)}.
Then, there exists $D>0$ such that for any relatively compact open set $G$
\begin{align}
{\E}_x\left[r_p^2(X_{t\land\tau_G})\right]
\leq r_p^2(x)+2(1+D)t\label{eq:quadratic1}
\end{align} 
holds for all~$x\in M$. 
\item Suppose {\bf (B2)}.
Then, there exists $D>0$ such that for any relatively compact open set $G$ 
\begin{align}
{\E}_x\left[r_p(X_{t\land\tau_G})-\log(1+r_p(X_{t\land\tau_G}))\right]
\leq r_p(x)-\log(1+r_p(x))+(1+D)t\label{eq:quadratic3}
\end{align} 
holds for all~$x\in M$. 
\item Suppose {\bf (B3)}.
Then, there exists $D>0$ such that for any relatively compact open set $G$
\begin{align}
{\E}_x\left[\log(1+r_p^2(X_{t\land\tau_G}))\right]
\leq \log(1+r_p^2(x))+2(1+D)t\label{eq:quadratic4}
\end{align} 
holds for all~$x\in M$. 
\end{enumerate}
\end{lem}
\begin{proof}[\bf Proof]

By It\^o's formula, we have
\begin{align*}
r_p^2(X_t)&=r_p^2(X_0)+2\sqrt{2}\int_0^tr_p(X_s)\d\beta_s+
2t\\
&\hspace{1cm}+2\int_0^t r_p(X_s)\Delta_Vr_p(X_s)\1_{M\setminus {\rm Cut}(p)}(X_s)\d s-2\int_0^t r_p(X_s)\d L_s,\\
r_p(X_t)-\log(1+r_p(X_t))&=r_p(X_0)-\log(1+r_p(X_0))\\
&\hspace{0.5cm}+\sqrt{2}\int_0^t\frac{r_p(X_s)\d \beta_s}{1+r_p(X_s)}+\int_0^t
\frac{\d s}{(1+r_p(X_s))^2}\\
&\hspace{1cm}+\int_0^t\frac{r_p(X_s)\Delta_Vr_p(X_s)}{1+r_p(X_s)}
\1_{M\setminus {\rm Cut}(p)}(X_s)\d s-\int_0^t\frac{r_p(X_s)\d L_s}{1+r_p(X_s)},\\
\log(1+r_p^2(X_t))=\log&(1+r_p^2(X_0))+2\sqrt{2}\int_0^t\frac{r_p(X_s)}{1+r_p^2(X_s)}\d \beta_s\\
&\hspace{-0.4cm}+2\int_0^t
\frac{1-r_p^2(X_s)}{(1+r_p^2(X_s))^2}\d s
+2\int_0^t\frac{r_p(X_s)\Delta_Vr_p(X_s)}{1+r_p^2(X_s)}
\1_{M\setminus {\rm Cut}(p)}(X_s)\d s\\
&\hspace{2cm}-\int_0^t\frac{2r_p(X_s)}{1+r_p^2(X_s)}\d L_s,\\
\end{align*}
$t\in[\,0,\,+\infty\,[$, 
${\P}_x$-a.s.~for $x\in M$.
From these formulae,  
we can obtain the conclusion. 
\end{proof}

\section{Lower bound}
Let $(X_t)_{t\geq0}$ be the $\Delta_V$-diffusion process on $M$.
In this section, we show a lower bound of ${\E}_x[{\color{black}{{\sf d}}}_N^{\,2}(u(X_{t\land\tau_{G_n}}),o)]$
for a reference point $o\in N$ and $V$-harmonic map $u:M\to N$ with the increasing sequence $\{G_n\}$ of relatively compact open sets satisfying $\overline{G_n}\subset G_{n+1}\subset M$ for $n\in\N$ and 
$M=\bigcup_{n=1}^{\infty}G_n$. 
The proof of the
lower estimate ${\E}_x[{\color{black}{{\sf d}}}_N(u(X_t),o)^2]\geq {\color{black}{{\sf d}}}_N(u(x),o)^2+|{\rm d}u|^2(x)t$ for
harmonic map $u:M\to N$ in terms of Brownian motion $(X_t)_{t\geq0}$ on $M$ under ${\Ric}_g\geq0$ in \cite{Staff:Liouville} is incomplete, because of the lack of 
the gradient estimate for harmonic map $u$ like Lemma~\ref{lem:lineargrowth}. In \cite{Staff:Liouville}, 
such a gradient estimate for harmonic function $u:M\to\R$ 
due to Li-Yau~\cite{LiYau} was 
only used.

\begin{lem}\label{lem:subharmograd}

Assume that ${\Ric}_V^m\geq0$ with $m\in\,[\,-\infty,\,0\,]\,\cup\,[\,n,+\infty\,]$ and
${\rm Sect}_N\leq0$. 
Let $u:M\to N$ be a $V$-harmonic map.
Suppose {\bf (G1)} for $u$  and {\bf (B3)}. 
Then, 
\begin{align}
|\d u|^2(x)\leq {\E}_x\left[|\d u|^2(X_t)\right]\label{eq:sumartingale}
\end{align}
holds for all $x\in M$.

\end{lem}
\begin{proof}[\bf Proof]
Applying Corollary~\ref{cor:BochnerIdentity} to the $V$-harmonic map $u$, we have
\begin{align*}
\Delta_V |\d u |^2\geq0\quad\text{ on }\quad M.
\end{align*}
Then It\^o's formula \eqref{eq:Ito} tells us that there exists a local  martingale $m_t^{|\d u|^2}$
such that
\begin{align*}
|\d u|^2(X_t)-|\d u|^2(X_0)=m_t^{|\d u|^2}+\int_0^t\Delta_V |\d u|^2(X_s)\d s,\qquad t\in\,[\,0,\,+\infty\,[
\end{align*}
holds $\P_x$-a.s.~for all $x\in M$.
In fact the local ${\P}_x$-martingale $m_t^{|\d u|^2}$
 is given by
\begin{align*}
m_t^{|\d u|^2}:=\sqrt{2}\int_0^t\left\langle U_s^{-1}\nabla|\d u|^2(X_s),\d B_s\right\rangle.
\end{align*}
Let $G$ be a relatively compact open set. 
We see that 
\begin{align*}
|\d u|^2(X_{t\land\tau_G})-|\d u|^2(X_0)=m_{t\land\tau_G}^{|\d u|^2}+\int_0^{t\land\tau_G}\Delta_V |\d u|^2(X_s)\d s 
\end{align*}
holds $\P_x$-a.s.~for all $x\in M$. 
Since ${\E}_x[m_{t\land\tau_G}^{|\d u|^2}]=0$, we have 
\begin{align*}
|\d u|^2(x)\leq {\E}_x\left[|\d u|^2(X_{t\land\tau_G})\right]
\end{align*}
for $x\in M$. 
Now we consider an increasing sequence $\{G_n\}$ of relatively compact open sets satisfying $\overline{G_n}\subset G_{n+1}\subset M$ for $n\in\N$ and 
$M=\bigcup_{n=1}^{\infty}G_n$.
We prove that 
\begin{align}
\sup_{n\in\N}{\E}_x\left[|\d u|^4(X_{t\land\tau_{G_n}})\right]<\infty\label{eq:uniformbdd}
\end{align}
holds for all $x\in M$. 
By Lemma~\ref{lem:lineargrowth} and {\bf (G1)} for $u$, we have {\bf (G1)} for $|\d u|$, hence there exists $A=A(p)>0$ such that 
\begin{align*}
|\d u|^4(x)\leq r_p^4(x)\quad \text{ for }\quad r_p(x)>A.
\end{align*}
Then 
\begin{align*}
{\E}_x\left[|\d u|^4(X_{t\land\tau_{G_n}})\right]&={\E}_x\left[|\d u|^4(X_{t\land\tau_{G_n}}):r_p(X_{t\land\tau_{G_n}})\leq A\right]\\&\hspace{2cm}+{\E}_x\left[|\d u|^4(X_{t\land\tau_{G_n}}):r_p(X_{t\land\tau_{G_n}})>A\right]\\
&\leq \sup_{z\in B_A(p)}|\d u|^4(z)+{\E}_x[r_p^4(X_{t\land\tau_{G_n}})].
\end{align*}
By \eqref{eq:quaternic}
\begin{align*}
{\E}_x\left[|\d u|^4(X_{t\land\tau_{G_n}})\right]&\leq 
\sup_{z\in B_A(p)}|\d u|^4(z)
+r_p(x)^4+4(3+D)\int_0^t {\mathfrak D}(s)\d s \\
&\hspace{1cm}+4De^{4Dt}\int_0^t\left(r_p^4(x)+4(D+3)\int_0^s{\mathfrak D}(u)\d u\right)e^{-4Ds}\d s
\end{align*}
holds for all $x\in M$. 
This implies \eqref{eq:uniformbdd}. Therefore $\{|\d u|^2(X_{t\land\tau_{G_n}})\}_{n=1}^{\infty}$ is 
uniformly ${\P}_x$-integrable for all $x\in M$. Since 
\begin{align*}
{\P}_x\left(|\d u|(X_t)=\lim_{n\to\infty}|\d u|(X_{t\land\tau_{G_n}})\right)=1\quad\text{ for \ all }\quad x\in M
\end{align*}
by {\color{black}{Corollary}}~\ref{cor:conservative} and 
Theorem~\ref{thm:tight}, $\{|\d u|(X_{t\land\tau_{G_n}})\}_{n=1}^{\infty}$ $L^2({\P}_x)$-converges to $|\d u|(X_t)$ for all $x\in M$, hence 
\begin{align*}
|\d u|^2(x)\leq \lim_{n\to\infty}{\E}_x[|\d u|^2(X_{t\land\tau_{G_n}})]={\E}_x[|\d u|^2(X_t)]
\end{align*}
for all $x\in M$. 
\end{proof}

\begin{lem}\label{lem:growthEst}
Assume that ${\Ric}_V^m\geq0$ for $m\in\,[-\infty,\,0\,]\,\cup\,[\,n,+\infty\,]$ and
$(N,h)$ is a Cartan-Hadamard manifold.  Let $u:M\to N$ be a $V$-harmonic map. 
Suppose {\bf (G1)} for $u$ and {\bf (B3)}.
Then, for any increasing sequence $\{G_n\}$ of relatively compact open sets satisfying $\overline{G_n}\subset G_{n+1}\subset M$ for $n\in\N$ and 
$M=\bigcup_{n=1}^{\infty}G_n$, 
\begin{align}
2t
|\d u|^2(x)
\leq\sup_{n\in \N}{\E}_x
\left[{\color{black}{{\sf d}}}_N^{\,2}(u(X_{t\land\tau_{G_n}}),o)\right]
\label{eq:growthEst}
\end{align} 
for all $x\in M$. 
\end{lem}
\begin{proof}[\bf Proof]
It\^o's formula \eqref{eq:Ito} tells us that there exists a local  martingale $m_t=m_t^{{\color{black}{{\sf d}}}_N^{\,2}(u,o)}$ 
such that
\begin{align*}
{\color{black}{{\sf d}}}_N^{\,2}(u(X_t),o)-{\color{black}{{\sf d}}}_N^{\,2}(u(X_0),o)=m_t+\int_0^t\Delta_V {\color{black}{{\sf d}}}_N^{\,2}(u(X_s),o)\d s,\qquad t\in\,[\,0,\,\infty\,[
\end{align*}
holds $\P_x$-a.s.~for all $x\in M$.
In fact $m_t$ is given by
\begin{align*}
m_t:=\sqrt{2}\int_0^t\left\langle U_s^{-1}\nabla {\color{black}{{\sf d}}}_N^{\,2}(u(X_s),o),\d B_s\right\rangle.
\end{align*}
Let $G$ be a relatively compact open set. 
By $\E_x[m_{t\land\tau_G}]=0$, we have
\begin{align}
\E_x\left[{\color{black}{{\sf d}}}_N^{\,2}(u(X_{t\land\tau_G}),o) \right]-{\color{black}{{\sf d}}}_N^{\,2}(u(x),o)=\E_x\left[\int_0^{t\land\tau_G}\Delta_V{\color{black}{{\sf d}}}_N^{\,2}(u,o)(X_s)\d s\right].\label{eq:EqBoth}
\end{align}
By \eqref{eq:sumartingale}, {\color{black}{Corollary}}~\ref{cor:conservative} and Theorem~\ref{thm:tight}, 
we can deduce that 
\begin{align*}
2t
|\d u|^2(x)
&\leq 2\int_0^t\E_x
[|\d u|^2(X_s)]\d s\\
&=2{\E}_x
\left[\int_0^t|\d u|^2(X_s)\d s\right]\\
&=2\sup_{n\in\N}\E_x
\left[\int_0^{t\land\tau_{G_n}}|\d u|^2(X_s)\d s \right]\\
&\leq\sup_{n\in\N}\E_x
\left[\int_0^{t\land\tau_{G_n}}\Delta_V{\color{black}{{\sf d}}}_N^{\,2}(u,o)(X_s)\d s \right]\quad \text{(by Lemma~\ref{lem:LapalcianDistaHarm})}\\
&\leq \sup_{n\in\N}\E_x
[{\color{black}{{\sf d}}}_N^{\,2}(u(X_{t\land\tau_{G_n}}),o)]-{\color{black}{{\sf d}}}_N^{\,2}(u(x),o)
\quad (\text{by \eqref{eq:EqBoth}})
\end{align*} 
for all $x\in M$. 
\end{proof}
\section{Proof of Theorem~\ref{thm:Liouville1}}
\begin{proof}[\bf Proof of Theorem~\ref{thm:Liouville1}]
Under ${\rm Ric}_V^m\geq0$, for each $i=1,2,3$, the condition $\text{\bf (A\text{\boldmath$i$})}$ implies $\text{\bf (B\text{\boldmath$i$})}$ by Theorem~\ref{thm:(A)0}.  
Suppose that $u:M\to N$ is a $V$-harmonic map having growth conditions. 
Since $o(\sqrt{a})$ (resp.~$o(\sqrt{\log a})$) implies $o(\sqrt{a-\log(1+a)})$ 
(resp.~$o(\sqrt{\log(1+a^2)}$), 
for any $\eps>0$, there exists $A>0$ such that for any $a>A$, 
\begin{align}
m_{{\color{black}{{\sf d}}}_N^{\,2}(u,o)}(a)\leq\eps \left\{\begin{array}{cl} a & \text{ under {\bf (G1)}}, \\ 
\sqrt{a-\log(1+a)} & \text{ under {\bf (G2)}}, 
 \\ \sqrt{\log(1+a^2)} & \text{ under {\bf (G3)}}. \end{array}\right.\label{eq:sublinearet}
\end{align}
Moreover, for $a\leq A$,
$m_{{\color{black}{{\sf d}}}_N^{\,2}(u,o)}(a)\leq m_{{\color{black}{{\sf d}}}_N^{\,2}(u,o)}(A)$. 
We find that 
\begin{align*}
{\bf E}_x
\left[{\color{black}{{\sf d}}}_N^{\,2}(u(X_{t\land\tau_{G_n}}),o)\right]&={\bf E}_x
\left[{\color{black}{{\sf d}}}_N^{\,2}
(u(X_{t\land\tau_{G_n}}),o):r_p(X_{t\land\tau_{G_n}})\leq A\right]\\&\hspace{0.3cm}
+{\bf E}_x
\left[{\color{black}{{\sf d}}}_N^{\,2}
(u(X_{t\land\tau_{G_n}}),o):r_p(X_{t\land\tau_{G_n}})>A\right]\\
&\leq \sup_{z\in B_A(p)}{\color{black}{{\sf d}}}_N^{\,2}(u(z),o)\\&\hspace{1cm}+{\bf E}_x
\left[{\color{black}{{\sf d}}}_N^{\,2}
(u(X_{t\land\tau_{G_n}}),o):r_p(X_{t\land\tau_{G_n}})>A\right]\\
&\hspace{-0.1cm}\stackrel{\eqref{eq:sublinearet}}{\leq} \sup_{z\in B_A(p)}{\color{black}{{\sf d}}}_N^{\,2}(u(z),o)\\
&\hspace{0.4cm}+\eps^2
\left\{\begin{array}{lc}{\bf E}_x
\left[r_p^2(X_{t\land\tau_{G_n}})\right] & \text{ under }\text{\bf (G1)}, \\ 
{\bf E}_x
\left[r_p(X_{t\land\tau_{G_n}})-\log(1+r_p(X_{t\land\tau_{G_n}}))\right]
 & \text{ under }\text{\bf (G2)}, \\ {\bf E}_x
 \left[\log(1+r_p^2(X_{t\land\tau_{G_n}}))\right]& \text{ under }\text{\bf (G3)}\end{array}\right.
\\
&\leq \sup_{z\in B_A(p)}{\color{black}{{\sf d}}}_N^{\,2}(u(z),o)\\
&\hspace{0.4cm}+\eps^2
\left\{\begin{array}{lc}r_p^2(x)+2(1+D)t & \text{ under }\text{\bf (B1)}, \\ 
r_p(x)-\log(1+r_p(x))+(1+D)t
 & \text{ under }\text{\bf (B2)}, \\ 
 \log(1+r_p^2(x))+2(1+D)t& \text{ under }\text{\bf (B3)}\end{array}\right.
\end{align*}
for all $x\in M$. 
In the last inequality, we use Lemma~\ref{lem:squaredistance*}. 
By Lemma~\ref{lem:growthEst}, 
\begin{align}
2t|\d u|^2(x)
&\leq \sup_{n\in\N} 
{\bf E}_x
\left[{\color{black}{{\sf d}}}_N^{\,2}(u(X_{t\land\tau_{G_n}}),o)\right]\notag\\
&\leq
\sup_{z\in B_A(p)}{\color{black}{{\sf d}}}_N^{\,2}(u(z),o)+\eps^2(r_p^2(x)+2(1+D)t)\label{eq:UL}
\end{align} 
for all $x\in M$ 
under {\bf(B{\boldmath$i$})} and the growth condition {\bf(G{\boldmath$i$})} for each $i=1,2,3$.  
Dividing \eqref{eq:UL} by $t$ and letting $t\to\infty$, 
we see 
\begin{align*}
2|\d u|^2(x)\leq 2\eps^2(1+D)
\end{align*}
for all $x\in M$. Since $\eps>0$ is arbitrary, we have 
$|\d u|^2(x)=0$ for all $x\in M$.
\end{proof}
\begin{remark}\label{rem:relaxation}
{\rm By Remark~\ref{rem:simplyconnected}, the conclusion of Theorem~\ref{thm:Liouville1} 
under {\bf(A{\boldmath$i$})} and {\bf(G{\boldmath$i$})} for each $i=1,2,3$ 
also holds if
${\rm Ric}_V^m\geq0$, ${\rm Sect}_N\leq0$ and $u:M\to N$ is a $V$-harmonic map satisfying 
${\rm Im}(u)\cap {\rm Cut}(o)=\emptyset$ for some point $o\in N$.
}
\end{remark}
\begin{remark}
{\rm In Theorem~\ref{thm:Liouville1}, under {\bf(A{\boldmath$i$})} and {\bf(G{\boldmath$i$})} for each $i=1,2,3$, 
if further the symmetric matrix $({\rm Hess}\,u(e_i,e_j))$ for the $V$-harmonic map $u$ 
admits a null eigenvalue with multiplicity $k\in\{1,2,\cdots,$ $ n-1\}$ independent of the choice of fibre, then the same conclusion holds
under ${\Ric}_V^m\geq0$ with $m\in]-\infty,\,k\,]$ and ${\rm Sect}_N\leq0$ with 
$\pi_1(N)=\{e\}$.
}
\end{remark}

\section{Proofs of Theorem~\ref{thm:Liouville2}, Corollaries~\ref{cor:ChenJostQiu}, \ref{cor:New} 
and Theorem ~\ref{thm:Liouville3}
}\label{sec:ProofsOfThmsCors}
Our proof relies on the convex geometry on the regular geodesic ball on complete Rimannian manifolds with positive upper curvature, first introduced by Kendall~\cite{Kend:probconvI,Kend:hemisphere}. In our situation, more precisely, there exists a convex function $\Phi: B_R(o) \times B_R(o) \rightarrow \R$, such that for any geodesic $\gamma$ in $B_R(o) \times B_R(o)$, there exists a positive constant $C$,
\begin{align}\label{convexity.func}
\frac{\d^2}{\d t^2}\Phi(\gamma(t)) \geq C\left(\frac{\d}{\d t}\Phi(\gamma(t))  \right)^{2}.
\end{align}
Moreover, $\Phi$ is a nonnegative and bounded function. It is also symmetric $\Phi(x, y) = \Phi(y, x)$, and vanishes only on its diagonal, i.e. $\Phi(x, y) = 0$ if and only if $x=y$. For the concrete construction of $\Phi(x, y)$, see \cite{Kend:hemisphere}.

\begin{lem}\label{lem:submartingaleProperty}
Fix $q \in B_R(o)$, then $\Phi(u(x),q)$ can be considered as a function on $M$.  Moreover, we have
\begin{align}\label{subharmonic}
\Delta_V\Phi( u(x),q) \geq 0.
\end{align}
\end{lem}

\begin{proof}[\bf Proof]
Recall that for harmonic map $u$, we have
\begin{align}\label{eqDelta}
\Delta_V\Phi(u(x),q) = \sum^{n}_{i=1} {\rm Hess}^{N} \Phi(\cdot,q)
(\d u(e_{i}), \d u(e_{i})).
\end{align}
Here ${\rm Hess}^{N}\Phi(\cdot,q)$ is the Hessian of $\Phi(\cdot,q)$ with respect to 
$(N,h)$. 
Notice that for a geodesic $\gamma(t) \in B_R(o)$, we have
$$
\left.\frac{\d^2}{\d t^2}\right|_{t = t_{0}}\Phi(\gamma(t),q)  =  {\rm Hess}^{N} \Phi(\cdot,q)
(\dot{\gamma}(t_{0}), \dot{\gamma}(t_{0})).
$$
Now we consider the geodesic $\gamma(t)$,  such that for $\gamma(t_{0}) = u(x)$ and $\dot{\gamma}(t_{0}) = X$, where $X \in T_{u(x)}N$.  Then,
\begin{align*}
\left.\frac{\d^2}{\d t^2}\right|_{t = t_{0}}\Phi(\gamma(t),q)  &=  {\rm Hess}^{N} \Phi(\cdot,q)(X, X), \\
\left.\frac{\d}{\d t}\right|_{t = t_{0}}\Phi(\gamma(t),q)  &= \langle \nabla^{N} \Phi(\cdot,q), X\rangle.
\end{align*}
Therefore by \eqref{convexity.func} and \eqref{eqDelta}, we derive that
\begin{align*}
\Delta_V\Phi(u(x),q) \geq C\sum^{n}_{i=1} \langle \nabla^{N} \Phi(u(x),q), \d u(e_{i}) \rangle^{2} \geq 0.
\end{align*}

\end{proof}

Now we are in the position to prove Theorem~\ref{thm:Liouville2}.

\begin{proof}[\bf Proof of Theorem~\ref{thm:Liouville2}]
Recall that ${\bf X}^V$ is conservative by {\color{black}{Corollary}}~\ref{cor:conservative}. 
According to \eqref{subharmonic}, we know that $\Phi(u,q)$ can be seen as a bounded $V$-subharmonic function on $M$. Then it suffice to prove a special case of Theorem~\ref{thm:Liouville1}, with $N = \R$ and $v(x) := \Phi(u(x),q)$ is a non-negative, bounded $V$-subharmonic function on $M$.

 As before, let $X_{t}$ be the diffusion on $M$ whose generator is $\Delta_{V}$ and  with initial position $X_{0} = x$. Now apply It\^o's formula for $v(X_{t})$, 
 we derive that
\begin{align*}
m_t:=v(X_{t}) - v(X_{0})-\int^{t}_{0}\Delta_{V}v(X_{s})\d s
\end{align*}
is a local (${\P}_x$-)martingale. 
Thus, by Lemma~\ref{lem:submartingaleProperty} $v(X_{t})$ is a non-negative, bounded, local submartingale. 
Let  $\{G_n\}$ the increasing sequence of relatively compact open sets satisfying $\overline{G_n}\subset G_{n+1}\subset M$ for $n\in\N$ and 
$M=\bigcup_{n=1}^{\infty}G_n$.
Then $n \mapsto v(X_{\tau_{G_{n}}})$ is a discrete time submartingale. By the martingale convergence theorem, there exists a unique limit of $v(X_{\tau_{G_{n}}})$, $\P_{x}$-a.s.

Now let ${\rm Im }(u) \subset B_R(o)$ be the image of $u$ and we may assume that it is a compact subset of $B_R(o)$.  Then for each $\omega \in \Omega$, there exists a subsequence $n_k=n_k(\omega)$ such that the limit
$$
\lim_{k \to \infty} u(X_{\tau_{G_{n_k}}}(\omega))=L(\omega)\in B_R(o) \subset N
$$
exists.

Therefore, we may conclude that for each $\omega \in\Omega$,
$$
\lim_{n \to\infty} v(X_{\tau_{G_{n}}}(\omega)) = \lim_{k \to\infty}\Phi(u(X_{\tau_{G_{n_k}}}(\omega)),q) = \Phi(L(\omega),q).
$$
Moreover, 
we also have
$$
\lim_{n \to\infty} v(X_{\tau_{G_{n}}} \circ \theta_t\,)  =  \Phi( L \circ \theta_t,q).
$$
On the other hand, for large $n\in\N$, $t<\tau_{G_n}$, hence $t+\tau_{G_n}\circ\theta_t=\tau_{G_n}$ 
$\P_x$-a.s. so that
$$
\lim_{n \to\infty} v(X_{\tau_{G_{n}}}) \circ \theta_t  = \lim_{n \to\infty} v(X_{t+\tau_{G_{n}}\circ\,\theta_t} )=\lim_{n\to\infty}v(X_{\tau_{G_n}})
=  \Phi( L ,q)
$$
holds $\P_x$-a.s.~for all $x\in M$. Thus
$$
 \Phi(L  \circ \theta_t,q) = \Phi( L ,q)
$$
holds for 
any $q \in B_R(o)$. Taking $q = L(\omega)$, we have
$$
\Phi(L  \circ \theta_t(\omega),L(\omega))=0\qquad\text{ for }\qquad t\in[\,0,+\infty\,[
$$
implying that $L  \circ \theta_t(\omega)=L(\omega)$. Therefore $L \in \bigcap_{0 < t < \infty} \mathscr{I}_{t}$, where $\mathscr{I}_{t}$ is the invariant $\sigma$-field of $\theta_{t}$.
We claim that $\mathscr{I}$ is trivial, leading to that $L$ is constant $L\equiv L_{\infty}$, $\P_{x}$-a.s..

To prove the triviality of $\mathscr{I}$, first notice that for any $A \in \mathscr{I}$, we have $\P_{x}(A\,|\,\mathscr{F}_t) = \P_{X_{t}}(A)$, 
meaning that the function $x \mapsto \P_{x}(A)$ is finely $V$-harmonic on $M$, i.e.
$t\mapsto \P_{X_t}(A)$ is a local (${\P}_x$-)martingale. 
Indeed, for any increasing sequence $\{G_n\}$ of relatively compact open sets satisfying $\overline{G_n}\subset G_{n+1}\subset M$ for $n\in \N$ and $M=\bigcup_{n=1}^{\infty}G_n$, 
$Y_t:=\P_{X_{t\land\tau_{G_n}}}(A)$ is a bounded martingale: for $s<t$
\begin{align*}
\E_x[Y_{t}\,|\,\mathscr{F}_s]&
=\E_x\left[\P_{X_{t\land \tau_{G_n}}}(A)\,\left|\,\mathscr{F}_s \right]\right.\\
&=\E_x\left[\P_x(A\,|\,\mathscr{F}_{t\land\tau_{G_n}})\,\left|\,\mathscr{F}_s \right]\right.\\
&=\P_x(A\,|\,\mathscr{F}_{s\land\tau_{G_n}})=Y_{s}.
\end{align*}
Meanwhile by taking $N = \R$ in Theorems 3.1 and 3.2 in Kendall's paper~\cite{Kend:probconvII}, we see that $x \mapsto \P_{x}(A)$ is in fact smooth on $M$. Indeed, $\P_{x}(A)$ is finely $V$-harmonic and the vector field $V$ is of course locally bounded. It is worth to point out that in the proof of \cite[Theorem 3.1]{Kend:probconvII}, we do not need  to assume the continuity of the map.

Therefore, the local (${\P}_x$-)martingale property and the smoothness  give rise to that $\P_{x}(A)$ is a $V$-harmonic function. By Theorem~\ref{thm:Liouville1}, we derive that $\E_{x}[\1_{A}] \equiv c$ for some constant $0 \leq c \leq 1$.

On the other hand, we have $c^{2} = c$. To see this, the invariance of $A$ and Markov property give rise to
\begin{align*}
c &= \E_{x}[\1_{A}] =  \E_{x}[ \E_{x}[\1^{2}_{A}\,|\, \mathscr{I}\,]\,] =  \E_{x}[\1_{A} \E_{x}[\1_{A}\,|\, \mathscr{I}\,]\,]  =  \E_{x}[\1_{A} \lim_{t \to \infty}\E_{x}[\1_{A}\,|\, \mathscr{I}_{t}\,]\,]\\
&=  \E_{x}[\1_{A} \lim_{t \to \infty}\E_{x}[\1_{A} \circ \theta_{t}\,|\, \mathscr{I}_{t}\,]\,] = \E_{x}[\1_{A} \lim_{t \to \infty}\E_{X_{t}}[\1_{A}]\,] = c^{2},
\end{align*}
implying that either $c = 0$ or $c=1$.

Thus,
\begin{align*}
\lim_{n \to\infty}\Phi(u(X_{\tau_{G_n}}),L_{\infty}) =\Phi(L(\omega),L_{\infty})=0.
\end{align*}

Consequently, submartingale property of $n\mapsto \Phi(u(X_{\tau_{G_n}}),L_{\infty})$ gives rise to
\begin{align*}
\Phi( u(x),L_{\infty})\leq \lim_{n \rightarrow \infty }\E_x\left[\Phi( u(X_{\tau_{G_n}}),L_{\infty}) \right] = 0,
\end{align*}
leading to  $\Phi( u(x),L_{\infty}) = 0$, $\P_x$-a.s.~for all $x\in M$, since $\Phi(x,y)$ is nonnegative. Therefore by the vanishing property, we have
$$
u \equiv L_{\infty}.
$$
\end{proof}
\begin{proof}[\bf Proof of Corollary~\ref{cor:ChenJostQiu}]
By Remark~\ref{rem:LaplacianCompaInfinity}, ${\rm Ric}_V^{\infty}\geq0$ implies $\text{\bf (B3)}$ and 
Theorem~\ref{thm:Liouville2} remains valid by replacing $\text{\bf (A3)}$ with $\text{\bf (B3)}$. 
Then we obtain the conclusion. 
\end{proof}

\begin{proof}[\bf Proof of Corollary~\ref{cor:New}]
By Corollary~\ref{cor:Vbounded}, 
${\rm Ric}_V^m\geq0$ with the boundedness of $V$ implies $\text{\bf (B3)}$ and 
Theorem~\ref{thm:Liouville2} remains valid by replacing $\text{\bf (A3)}$ with $\text{\bf (B3)}$. 
Then we obtain the conclusion. 
\end{proof}

To prove Theorem~\ref{thm:Liouville3}, we need the following:
\begin{lem}\label{lem:recurrent}
Suppose that the $\Delta_V$-diffusion process ${\bf X}^{V}=(\Omega,X_t,{\P}_x)$ is recurrent. Then every bounded finely $V$-subharmonic {\rm(}or $V$-superharmonic{\rm)} function is constant. In particular, every bounded smooth $V$-subharmonic {\rm(}or $V$-superharmonic{\rm)} function is constant.
Conversely, if every bounded finely $V$-subharmonic function is constant, then ${\bf X}^{V}$ is recurrent.
\end{lem}
\begin{proof}[\bf Proof]
Suppose that ${\bf X}^{V}$ is recurrent.
The recurrence of ${\bf X}^{V}$ implies ${\P}_x(\zeta=+\infty)=1$ for $x\in M$ (see \cite[(3.4) Lemma]{Get:TranRec}).
First note that any smooth bounded $V$-subharmonic function is finely $V$-subharmonic. This is easily confirmed by way of It\^o's formula.
Let $f$ be a bounded finely $V$-subharmonic (resp.~$V$-superharmonic) function, then $g:=\|f\|_{\infty}-f$ (resp.~$g=\|f\|_{\infty}+f$) is a non-negative
finely $V$-superharmonic function.
Let $\{G_n\}$ be the increasing sequence of relatively compact open sets 
satisfying $\overline{G_n}\subset G_{n+1}\subset M$ for $n\in N$ and $M=\bigcup_{n=1}^{\infty}G_n$. 
Then $t\mapsto g(X_{t\land\tau_{G_n}})$ is a bounded supermartingle, hence ${\E}_x[g(X_t):t<\tau_{G_n}]\leq g(x)$ for any $t>0$ and $x\in M$.
Letting $k\to\infty$, we have ${\E}_x[g(X_t)]\leq g(x)$ for any $t>0$ and $x\in M$.
Since $g$ is finely continuous, $g$ is an excessive function with respect to ${\bf X}^{V}$. Therefore we obtain the first conclusion. Conversely suppose that
any bounded finely $V$-subharmonic function is constant.
Let $g$ be a bounded excessive function with respect to ${\bf X}^{V}$.
Then
$\E_x[g(X_{t\land \tau_{G_n}})]\leq g(x)$ for all $t>0$ and $x\in M$ by
\cite[Chapter II, (2.8) Proposition]{BG:Markov}.
From this, we can deduce that
$t\mapsto g(X_{t\land \tau_{G_n}})$ is a bounded supermartingale, that is,
$g$ is finely $V$-superharmonic.
By assumption, $g$ is constant. For general excessive function $g$,
for each $\ell>0$, $g\land \ell$ is a bounded excessive function and it is a constant. If $g$ takes different values $g(x_1)$ and $g(x_2)$ with $g(x_1)<g(x_2)$, then $g(x)\land\left(\frac{g(x_1)+g(x_2)}{2}\right)$ has the constant value $c$.
Then one can obtain $c=g(x_1)=\frac{g(x_1)+g(x_2)}{2}$ so that $g(x_1)=g(x_2)$.
This is a contradiction. Therefore $g$ is constant, hence ${\bf X}^{V}$ is recurrent.
\end{proof}
\begin{proof}[\bf Proof of Theorem~\ref{thm:Liouville3}]
By Lemma~\ref{lem:recurrent}, $\Phi(u,o)$ is constant.
Thanks to the concrete expression of $\Phi$ shown in \cite[Theorem~3]{Kend:hemisphere}, ${\color{black}{{\sf d}}}_N(u,o)$ is constant. Then
one can conclude that $u(M)\subset \partial B_r(o)\subset B_R(o)$ for some $r>0$.
Fix an $x_0\in M$.
Since $B_R(o)\cap {\rm Cut}(o)=\emptyset$, $u(x_0)$ and $o$ can be joined by a unique minimal geodesic in $B_R(o)$. On this geodesic arc, we can choose a point $\hat{o}\in B_R(o)$ which is different from $o$ such that we can find another geodesic ball $B_{\hat{r}}(\hat{o})$ satisfying $B_r(o)\subset B_{\hat{r}}(\hat{o})\subset B_R(o)$.
Applying Lemma~\ref{lem:submartingaleProperty} to  $\Phi(u,\hat{o})$, it is constant
so that ${\color{black}{{\sf d}}}_N(u,\hat{o})\equiv \tilde{r}$, that is, $u(M)\subset \partial B_{\tilde{r}}(\hat{o})$ for some $\tilde{r}>0$.  Then we see
$\tilde{r}={\color{black}{{\sf d}}}_N(u(x_0),\hat{o})=r-{\color{black}{{\sf d}}}_N(o,\hat{o})$. Obviously, $u(x_0)\in \partial B_r(o)\cap\partial B_{\tilde{r}}(\hat{o})$. Suppose there exists a different
$u(x_1)\ne u(x_0)$ such that $u(x_1)\in \partial B_r(o)\cap\partial B_{\tilde{r}}(\hat{o})$. Since $u(x_1)\ne u(x_0)$, $u(x_1)$ does not lie on the minimal geodesic joining $o$ and $u(x_0)$. Then $r={\color{black}{{\sf d}}}_N(u(x_1),o)<{\color{black}{{\sf d}}}_N(u(x_1),\hat{o})+{\color{black}{{\sf d}}}_N(\hat{o},o)=
\tilde{r}+{\color{black}{{\sf d}}}_N(\hat{o},o)$. This is a contradiction. Hence $\partial B_r(o)\cap\partial B_{\tilde{r}}(\hat{o})=\{u(x_0)\}$,
consequently, $u(M)\subset \partial B_r(o)\cap\partial B_{\tilde{r}}(\hat{o})=\{u(x_0)\}$ implies $u(M)=\{u(x_0)\}$.
\end{proof}

\providecommand{\bysame}{\leavevmode\hbox to3em{\hrulefill}\thinspace}
\providecommand{\MR}{\relax\ifhmode\unskip\space\fi MR }
\providecommand{\MRhref}[2]{%
  \href{http://www.ams.org/mathscinet-getitem?mr=#1}{#2}
}
\providecommand{\href}[2]{#2}

\end{document}